\documentclass[12pt,twoside]{amsart}
\usepackage{a4wide, amsmath,amsbsy,amsfonts,amssymb,
stmaryrd,amsthm,mathrsfs,graphicx, amscd, tikz-cd, cancel}
\usepackage[normalem]{ulem}
\usepackage{soul}
\usetikzlibrary{matrix,arrows,decorations.pathmorphing}
\usepackage[active]{srcltx}
\usepackage[
	hypertexnames=false,
	hyperindex,
	pagebackref,  
	pdftex,
	breaklinks=true,
	bookmarks=false,
	colorlinks,
	linkcolor=blue,
	citecolor=red,
	urlcolor=red,
]{hyperref}
\usepackage{glossaries}
\makeindex

\usepackage[all]{xy}
\usepackage{subfig}
\usepackage{tikz}
\usetikzlibrary{shapes,arrows,shadows}
\usetikzlibrary{decorations.markings}
\usepackage{color}
\usepackage[normalem]{ulem}
\usepackage{hyperref}
\usepackage{wrapfig}
\usetikzlibrary{arrows}
\usepackage{pinlabel}
\long\def\symbolfootnote[#1]#2{\begingroup%
\def\thefootnote{\fnsymbol{footnote}}\footnote[#1]{#2}\endgroup}

\newtheorem{theorem}{Theorem}[section]

\newtheorem{proposition}[theorem]{Proposition}
\newtheorem{corollary}[theorem]{Corollary}
\newtheorem{lemma}[theorem]{Lemma}

\theoremstyle{definition}

\newtheorem{definition}[theorem]{Definition}
\newtheorem{remark}[theorem]{Remark}

\newtheorem*{namedtheorem}{\theoremname}
\newcommand{\theoremname}{testing}

\DeclareMathSymbol{\Alpha}{\mathalpha}{operators}{"41}
\DeclareMathSymbol{\Beta}{\mathalpha}{operators}{"42}
\DeclareMathSymbol{\Epsilon}{\mathalpha}{operators}{"45}
\DeclareMathSymbol{\Zeta}{\mathalpha}{operators}{"5A}
\DeclareMathSymbol{\Eta}{\mathalpha}{operators}{"48}
\DeclareMathSymbol{\Iota}{\mathalpha}{operators}{"49}
\DeclareMathSymbol{\Kappa}{\mathalpha}{operators}{"4B}
\DeclareMathSymbol{\Mu}{\mathalpha}{operators}{"4D}
\DeclareMathSymbol{\Nu}{\mathalpha}{operators}{"4E}
\DeclareMathSymbol{\Omicron}{\mathalpha}{operators}{"4F}
\DeclareMathSymbol{\Rho}{\mathalpha}{operators}{"50}
\DeclareMathSymbol{\Tau}{\mathalpha}{operators}{"54}
\DeclareMathSymbol{\Chi}{\mathalpha}{operators}{"58}
\DeclareMathSymbol{\omicron}{\mathord}{letters}{"6F}
\newcommand{\R}{\mathbb{R}}
\newcommand{\C}{\mathbb{C}}
\newcommand{\N}{\mathbb{N}}
\newcommand{\Z}{\mathbb{Z}}

\newcommand{\Q}{\mathbb{Q}}
\newcommand{\I}{\mathbb{I}}

\newcommand{\td}{\tilde}

\def\Aut{\operatorname{Aut}}

\def\SL{\operatorname{SL}}
\def\PSL{\operatorname{PSL}}

\def\Inn{\operatorname{Inn}}
\def\inn{\operatorname{inn}}

\def\Out{\operatorname{Out}}
\def\Map{\operatorname{Map}}

\def\Hom{\operatorname{Hom}}

\def\GT{\widehat{\operatorname{GT}}}
\def\PG{\mathrm{P}\Gamma}

\def\Star{\operatorname{Star}}
\def\Link{\operatorname{Lk}}

\def\Spec{\operatorname{Spec}}
\def\Fix{\operatorname{Fix}}
\def\P{{\mathbb P}}
\def\H{{\mathbb H}}
\def\M{{\mathbb M}}
\def\F{{\mathbb F}}

\def\cI{{\mathcal I}}

\def\cT{{\mathcal T}}

\def\cF{{\mathcal F}}
\def\hF{{\widehat{\operatorname{F}}}}

\def\cG{{\mathcal G}}

\def\cM{{\mathcal M}}

\def\bM{{\overline{\mathcal M}}}

\def\dd{\partial}
\def\a{{\alpha}}

\def\l{{\lambda}}
\def\L{{\Lambda}}

\def\Sg{{\Sigma}}
\def\g{{\gamma}}
\def\G{{\Gamma}}
\def\kG{{\check\Gamma}}
\def\kPG{\mathrm{P}{\check\Gamma}}
\def\hG{{\widehat\Gamma}}
\def\hPG{\mathrm{P}{\widehat\Gamma}}

\def\s{{\sigma}}
\def\u{{\upsilon}}

\def\kC{{\check  C}}
\def\hC{{\widehat{C}}}
\def\hI{{\hat{\mathrm{I}}}}

\def\hpi{{\widehat{\pi}}}

\def\ssm{\smallsetminus}
\def\ol{\overline}
\def\wh{\widehat}
\def\wt{\widetilde}
\def\ul{\underline}
\def\ra{\rightarrow}
\def\hookra{\hookrightarrow}
\def\co{\colon\thinspace}

\begin{document}

\title{Automorphisms of the procongruence pants complex}
 
 \author[M. Boggi]{Marco Boggi}
\address{UFF - Instituto de Matem\'atica e Estat\'{\i}stica - Niter\'oi - RJ 24210-200, Brazil.}
\email{marco.boggi@gmail.com}

\author[L. Funar]{Louis Funar}
\address{Institut Fourier,
Laboratoire de Mathematiques UMR 5582, Universit\'e Grenoble Alpes, CS 40700, 38058 Grenoble, France}
\email{louis.funar@univ-grenoble-alpes.fr} 

\begin{abstract}    
We show that every automorphism of the congruence completion of the extended mapping class group that preserves the set of conjugacy 
classes of procyclic groups generated by Dehn twists is inner, and that its automorphism group is naturally isomorphic to  the 
automorphism group of the procongruence pants complex. In the genus-zero case, we prove the stronger result that all automorphisms 
of the profinite completion of the extended mapping class group are inner.

\bigskip

\noindent {\bf AMS Math Classification:} 14G32, 11R32, 20F34, 14D23, 57M10.

\end{abstract}

\maketitle

\section{Introduction}
Let $S=S_{g,n}$\index{$S_{g,n}$, a surface of genus $g$ with $n$ punctures} be a closed orientable surface of genus $g(S)=g$ with $n(S)=n$ punctures. 
We will assume that $S$ has negative Euler characteristic:
$\chi(S)=2-2g-n<0$. Let  $\G^{\pm}(S)$\index{$\G^{\pm}(S)$, the extended mapping class group of $S$} be the {\em extended mapping class group} of $S$, namely the group of isotopy classes of
diffeomorphisms of $S$, and $\G(S)$\index{$\G(S)$ the mapping class group of $S$} be the {\em mapping class group} of $S$, i.e.\ the subgroup of $\G^{\pm}(S)$ consisting of 
the mapping classes which preserve the orientation of $S$. We denote respectively by $\PG(S)$\index{$\PG(S)$, the pure mapping class group of $S$} 
and $\PG^{\pm}(S)$\index{$\PG^{\pm}(S)$, the pure extended mapping class group of $S$} 
the \emph{pure} mapping class group and the \emph{pure} extended mapping class group, namely the subgroups of $\G(S)$ and $\G^{\pm}(S)$ 
consisting of those mapping classes which fix pointwise the punctures.

Let $\cM(S)$\index{$\cM(S)$, the moduli stack of smooth curves whose complex models are diffeomorphic to $S$} 
be the moduli stack of smooth curves whose complex models are diffeomorphic to $S$. This is a smooth Deligne--Mumford 
(briefly DM) stack over $\Spec\Z$ with the property that the topological fundamental group of $\cM(S)_\C:=\cM(S)\times\Spec\C$ identifies 
with $\G(S)$. 

For $S=S_{g,n}$, we will sometimes also denote $\cM(S)$ by $\cM_{g,[n]}$.
\index{$\cM_{g,[n]}$, the moduli stack of smooth curves of genus $g$ with $n$ \emph{unordered} punctures}
We will instead denote by $\cM_{g,n}$
\index{$\cM_{g,n}$, the moduli stack of smooth curves of genus $g$ with $n$ \emph{ordered} punctures} 
the \'etale covering of $\cM_{g,[n]}$ obtained fixing an order on the punctures of $S$. The topological fundamental group of $(\cM_{g,n})_\C$ 
then identifies with $\PG(S_{g,n})$. We let $d(S):=\dim\cM(S)=3g-3+n$\index{$d(S)$, modular dimension of $S$} and call this number the \emph{modular dimension of $S$}.

Ivanov (cf.\ \cite{I0} and \cite{I2}) proved that all automorphisms of the extended mapping class group  $\G^{\pm}(S)$ are inner for $g(S)\geq 3$ 
and for $g(S)\geq 2$, $n(S)\geq 1$. McCarthy (cf.\ \cite{McCarthy}) showed that this is still true, for $g(S)=2$ and $n(S)=0$, if we restrict to those 
automorphisms which preserve the set of conjugacy classes of Dehn twists. Korkmaz (cf.\ \cite{Korkmaz}) extended Ivanov's result to the case $g(S)=1$, 
$n(S)\geq 3$ and $g(S)=0$, $n(S)\geq 5$. The case $g(S)=1$, $n(S)=2$ was settled by Luo in \cite{Luo}, where he also gave a new proof for all genera.

The proof of all the above results is based on the study of the \emph{complex of curves} $C(S)$\index{$C(S)$, the curve complex of $S$}. 
This is the abstract simplicial complex of 
dimension $d(S)-1$ whose simplices are the sets of isotopy classes of essential simple closed curves on $S$ which admit disjoint representatives 
(such sets are called \emph{multicurves}). 
There is a natural action of the extended mapping class group $\G^{\pm}(S)$ on $C(S)$ and the above results are obtained  showing that 
all automorphisms of $C(S)$ are induced by this action. In fact, this action corresponds to the inner action of $\G^{\pm}(S)$ on the
set of Dehn twists, which are parametrized by the vertices of $C(S)$, and, if we denote by $\Aut^\I(\G^{\pm}(S))$ the group of automorphisms of 
$\G^{\pm}(S)$ which preserve the set of conjugacy classes of Dehn twists, there is a series of monomorphisms:
\[\Inn(\G^{\pm}(S))\subseteq\Aut^\I(\G^{\pm}(S))\hookra\Aut(C(S)),\]
which, for $d(S)>1$, with the only exception of the case $g(S)=1$ and $n(S)=2$, are all shown to be isomorphisms.
The identity $\Inn(\G^{\pm}(S))=\Aut(\G^{\pm}(S))$ then follows from the fact that all automorphisms of $\G^{\pm}(S)$, for $d(S)>1$, 
with few low-genus exceptions, preserve the set of conjugacy classes of Dehn twists.

In \cite{Margalit}, Margalit determined the automorphism group of a related complex, the so-called \emph{pants graph} $C_P(S)$.\index{$C_P(S)$, 
the pants complex of $S$}
This is defined as follows: the vertices of $C_P(S)$ are  pants decompositions (i.e.\ maximal multicurves) of $S$ and correspond to facets
(simplices of highest dimension $d(S)-1$) of $C(S)$. Two vertices $\ul \a=(\alpha_0,\alpha_1,\ldots,\alpha_{d(S)-1})$ and 
$\ul \a'=(\alpha'_0,\alpha'_1,\ldots,\alpha'_{d(S)-1})$ are connected by an edge if they differ by an \emph{elementary move}, that is:
the multicurves $\ul \a$ and $\ul \a'$ have $d(S)-1$ elements in common, so that, up to relabeling $\a_i=\a'_i$, for $i=1,\ldots, d(S)-1$,
and the surface $S'$ obtained cutting $S$ along the curves $\a_i$, for $i>0$, is a surface of modular dimension $1$, i.e.\ $S'=S_{1,1}$ or $S'=S_{0,4}$.
Then, $\a_0$ and $\a'_0$, which are supported on $S'$, should intersect in a minimal way, i.e.\ they 
have geometric intersection number $1$, in the first case, and $2$ in the second case. 

Margalit's results (cf.\ \cite[Theorem~1 and 2]{Margalit}) then imply that, for $d(S)>1$, 
the natural action of $\G^{\pm}(S)$ on $C_P(S)$ induces a series of isomorphisms:
\begin{equation}\label{Margalitiso}
\Inn(\G^{\pm}(S))\cong\Aut^\I(\G(S))\cong\Aut(C_P(S)).
\end{equation}

When the center $Z(\G(S))$ is trivial, that is, for $d(S)>1$ and $S\neq S_{1,2}$ and $S_2$, there is also a series of isomorphisms:
\begin{equation}\label{nonoriso1}
\Aut(\G^\pm(S))\cong\Aut^\I(\G^\pm(S))\cong\Aut^\I(\G(S))\cong\Aut(\G(S)).
\end{equation}

For $S=S_{1,1}, S_{1,2}$ or $S_2$, the center $Z(\G(S))=Z(\G^\pm(S))$ is instead generated by the hyperelliptic involution $\u$ and 
so $Z(\G(S))\cong\Z/2$. In this case, there are two epimorphisms $\xi\co\G^\pm(S)/Z(\G(S))\to Z(\G(S))$ and $\zeta\co\G^\pm(S)/Z(\G(S))\to Z(\G(S))$,
where the latter is induced by the orientation character (so that its restriction to $\G(S)/Z(\G(S))$ is trivial), which determine
the automorphisms $\exp\xi$ and $\exp\zeta$ of $\G^\pm(S)$, defined by the assignments $x\mapsto x\cdot\xi(\bar x)$ and $x\mapsto x\cdot\zeta(\bar x)$,
where $\bar x$ denotes the image of $x$ in the quotient $\G^\pm(S)/Z(\G(S))$. In these cases, we have the isomorphisms (cf.\ Section~\ref{exAut}): 
\begin{equation}\label{nonoriso2}
\Aut(\G^\pm(S))\cong\Aut^\I(\G^\pm(S))\times\langle\exp\xi\rangle\hspace{0.5cm}\mbox{and}\hspace{0.5cm}
\Aut^\I(\G^\pm(S))=\Aut^\I(\G(S))\times\langle\exp\zeta\rangle.
\end{equation}

In the paper \cite{BF}, we proved a partial analogue of the series of isomorphisms~\eqref{Margalitiso} in the setting of 
\emph{procongruence mapping class groups} which we now proceed to define.

The \emph{profinite} mapping class groups $\hG^\pm(S)$, $\hG(S)$,  $\hPG^\pm(S)$ and  $\hPG(S)$ are defined to be
the profinite completions of $\G^\pm(S)$, $\G(S)$,  $\PG^\pm(S)$ and  $\PG(S)$, respectively.
\index{$\hG(S)$, profinite completion of $\G(S)$}
\index{$\hPG^\pm(S)$, profinite completion of $\PG^\pm(S)$}
\index{$\hPG(S)$, profinite completion of $\PG(S)$}

The \emph{procongruence} mapping class groups $\kG^\pm(S)$, $\kG(S)$,  $\kPG^\pm(S)$ and  $\kPG(S)$\index{$\kG^\pm(S)$, 
procongruence completion of $\G^\pm(S)$}\index{$\kG(S)$, procongruence completion of $\G(S)$}
\index{$\kPG^\pm(S)$, procongruence completion of $\PG^\pm(S)$}
\index{$\kPG(S)$, procongruence completion of $\PG(S)$}
are the images of $\hG^\pm(S)$, $\hG(S)$,  $\hPG^\pm(S)$ and  $\hPG(S)$, respectively, in the profinite group $\Out(\wh{\pi_1(S)})$, 
where, for an abstract group $G$, we denote by $\wh{G}$\index{$\wh{G}$, the profinite completion of the group $G$} its profinite completion. 
It is well known that the natural homomorphism from each of the above mapping class groups to either its profinite or procongruence completion is injective
(both statements immediately follow from the proof of \cite[Theorem~2 and~3]{Grossman}). We then identify the abstract
groups with their images in the corresponding profinite groups. 

The profinite and the procongruence completions of the mapping class group coincide 
for $g(S)\leq 2$ (cf.\ \cite{Asada}, \cite{McReynolds}, \cite{Hyp} and \cite{CongTop}), while, for $g(S)\geq 3$, this is an open problem. 
For this reason, we will stick with the procongruence completion, since, in contrast with the profinite completion, 
some basic combinatorial properties are known, thanks to the results contained in \cite{Boggi} and \cite{CongTop}. 

In analogy with the topological case, for the study of the automorphism group of procongruence mapping class groups, it is useful to introduce
the {\em procongruence curve complex} $\kC(S)$.\index{$\kC(S)$, the  procongruence curve complex of $S$}
This is a \emph{simplicial profinite complex} (cf.\ \cite[Definition~3.2]{Boggi}) of dimension
$d(S)-1$, naturally associated to the procongruence completion of the  mapping class group and endowed with a natural continuous action of $\kG(S)$. 
Naively, $\kC(S)$ can be described as the inverse limit of the quotients of $C(S)$ by the action of congruence levels of $\G(S)$, 
that is, finite index subgroups of $\G(S)$ which are open for the congruence topology. 

However, \emph{this is not how $\kC(S)$ is actually defined}, and for the precise (rather technical) definition, we urge the reader 
to look at \cite[Section~4]{Boggi} or \cite[Section~4.6]{BF}. In particular, $\kC(S)$ is a genuine abstract simplicial complex so that, for instance, 
all the usual definitions of star and link of a simplex in a simplicial complex make perfect sense for it and do not need an ad hoc definition.

The set of \emph{profinite Dehn twists} of $\kPG(S)$ is the closure, inside this group, of the set of Dehn twists of $\PG(S)\subset\kPG(S)$. 
These elements were first introduced in \cite{B1} for profinite mapping class groups and studied extensively in \cite{Boggi} for procongruence
mapping class groups. A similar and obviously related notion for \'etale fundamental groups of configuration spaces 
of points on algebraic curves has been later introduced in \cite{HM}.

The key property of $\kC(S)$ is then that its set of $k$-simplices parameterizes the \emph{inertia groups} of $\kPG(S)$ (cf.\ \cite[Theorem~6.9]{Boggi}): 

\begin{definition}\label{inertiagroups}For a simplex $\s\in\kC(S)$, we define $\hI_\s$\index{$\hI_\s$, inertia group} to be the closed abelian subgroup 
of $\kPG(S)$ topologically generated by the profinite Dehn twists parameterized by the vertices of $\s$. This is called the \emph{inertia group} associated to $\s$.
\end{definition}

Note that the natural action of $\kG^\pm(S)$ on $\kC(S)$ corresponds to the conjugacy action of $\kG^\pm(S)$ on the profinite set of inertia groups:

\begin{definition}\label{inertiapreserving} Let $\Aut^{\I}(\kPG(S))$\index{$\Aut^{\I}(\kPG(S))$, the group of automorphisms which preserve the set of conjugacy classes of procyclic inertia groups} 
(resp.\ $\Aut^\I(\kG(S))$\index{$\Aut^{\I}(\kG(S))$, the group of automorphisms which preserve the set of conjugacy classes of procyclic inertia groups} and
$\Aut^{\I}(\kG^\pm(S))$\index{$\Aut^{\I}(\kG^\pm(S))$, the group of automorphisms which preserve the set of conjugacy classes of procyclic inertia groups}) be the subgroup of $\Aut(\kPG(S))$ 
(resp.\ $\Aut(\kG(S))$ and $\Aut^{\I}(\kG^\pm(S))$) consisting of those automorphisms which preserve the set of conjugacy classes of procyclic inertia groups. 
\end{definition}

\begin{remark}From \cite[Proposition~7.2]{BF}, it follows that the elements of $\Aut^{\I}(\kPG(S))$, $\Aut^\I(\kG(S))$ and $\Aut^{\I}(\kG^{\pm}(S))$ preserve the
set of conjugacy classes of \emph{all} inertia groups.
\end{remark}

There is then a series of natural monomorphisms:
\[\Inn(\kG^{\pm}(S))\subseteq\Aut^\I(\kG(S))\hookra\Aut^\I(\kPG(S))\hookra\Aut(\kC(S)).\]
Unlike in the topological case, however, this is not going to be a series of isomorphisms. A procongruence analogue 
of the pants graph $C_P(S)$ turns out to be more useful here.
 
The {\em procongruence pants complex $\kC_P(S)$}\index{$\kC_P(S)$, the procongruence pants complex of $S$} is the $1$-dimensional
simplicial profinite complex which realizes the inverse limit of the quotients of $C_P(S)$ by the action of congruence levels of $\G(S)$
(cf.\ \cite[Section~6.2]{BF}, for the precise definition). It is also endowed, by definition, with
a natural continuous action of $\kG^\pm(S)$. The profinite set of vertices of $\kC_P(S)$ identifies with the profinite set of facets of $\kC(S)$ and 
so it parameterizes the profinite set $\{\hI_\s|\, \s\in\kC(S)_{d(S)-1}\}$ of maximal abelian subgroups of $\kPG(S)$ topologically generated 
by profinite Dehn twists. 
The natural continuous action of $\kG^\pm(S)$ on $\kC_P(S)$ is then induced by the conjugacy action of $\kG^\pm(S)$ on this profinite set.

The main result of the paper is an analogue of the series of isomorphisms~\eqref{Margalitiso} and~\eqref{nonoriso1} 
and of the identities~\eqref{nonoriso2}:

\begin{theorem}\label{autpants}For  $S$ a connected hyperbolic surface, we have:
\begin{enumerate}
\item There is an isomorphism $\Inn(\kG^{\pm}(S))\cong\Aut(\kC_P(S))$.
\item For $d(S)>1$ and $S\neq S_{1,2},S_2$, there is a series of natural isomorphisms:
\[\Inn(\kG^{\pm}(S))=\Aut^{\I}(\kG^{\pm}(S))\cong\Aut^{\I}(\kPG^{\pm}(S)).\] 
\item For $S= S_{1,2}$ or $S_2$, there holds:
\[\Aut^{\I}(\kG^{\pm}(S))=\Inn(\kG^{\pm}(S))\times\langle\exp\zeta\rangle.\] 
\end{enumerate}
\end{theorem} 

\begin{remark}\label{I-condition}
The condition on automorphisms in Theorem~\ref{autpants} is slightly more restrictive than the one considered
in \cite{BF}, where we denoted by $\Aut^\ast(\kG(S))$ the subgroup of $\Aut(\kG(S))$ consisting of those automorphisms which preserve the
conjugacy classes of \emph{decomposition subgroups of} $\kG(S)$ (cf.\ \cite[Definition~7.1]{BF}). It is not difficult to see that there is an inclusion
$\Aut^{\I}(\kG(S))\subseteq\Aut^\ast(\kG(S))$ which for $d(S)>1$ and $Z(\kG(S))=\{1\}$ is indeed an equality but otherwise is strict.
\end{remark}

Throughout this article, we will make use of several results established in our former paper \cite{BF}, 
which studies the procongruence curve and pants complexes and has a 
large overlap with the article \cite{Lochak}, for the reasons explained there. As the referee 
pointed out to us, other notions of Dehn twists and curve complexes  
appeared in \cite{Gropper,Kumpitsch,Wilkes}, both in the profinite and pro-$p$ context.

Let us also observe, answering another of the referee's queries, that the isomorphism $\Inn(\kG^{\pm}(S))\cong\Aut(\kC_P(S))$ of 
(i), Theorem~\ref{autpants} is entirely different from the (only apparently) similar isomorphism stated in \cite[Theorem~7.1]{Lochak}. 
In fact, in Lochak's statement the procongruence graph $\kC_P(S)$ is replaced by the pro-object $\kC_P(S)_{st}$ 
determined by all orbifold quotients of the \emph{2-dimensional pants complex} by finite index subgroups of $\G(S)$ which are open 
for the congruence topology. Now, if $X$ denotes the 2-dimensional pants complex and $X^\l$  one of such orbifold quotients, since $X$ is simply connected, 
every automorphism of $X^\l$ lifts to an automorphism of $X$, so that \cite[Theorem~7.1]{Lochak} is a formal consequence (cf.\ \cite[Section~7.4]{Lochak}) 
of Margalit's rigidity theorem (cf.\ \cite[Theorem~1]{Margalit}).

Note that the first item of Theorem~\ref{autpants} is just an improvement, though a substantial one, of \cite[Theorem~8.1]{BF}. 
On the other hand, the isomorphisms in the second and third items of Theorem~\ref{autpants} are surprising new results 
and should be regarded as the main achievement of this paper. In fact, in the topological case, for a connected hyperbolic surface $S$ 
such that $d(S)>1$ and $S\neq S_{1,2},S_2$, there is a series of identities: 
\[\Aut^{\I}(\G(S))=\Inn(\G^{\pm}(S))=\Aut^{\I}(\G^\pm(S)).\] 
However, there is  no analogous result for the procongruence mapping class group $\kG(S)$, since its automorphism group $\Aut^{\I}(\kG(S))$ 
is quite far from only consisting of the automorphisms induced by restriction of inner automorphisms of $\kG^{\pm}(S)$. 
As a matter of fact, the outer automorphism group $\Out^{\I}(\kG(S))$\index{$\Out^{\I}(\kG(S))$, the group of outer automorphisms of $\kG(S)$ which preserve the conjugacy classes of procyclic inertia groups} 
contains a copy of the absolute Galois group of the rationals $G_\Q$ (cf.\ \cite[Corollary~7.6]{Boggi}).



A first interesting application of Theorem~\ref{autpants} is the following. Since $\Inn(\kG^{\pm}(S))$ identifies with a subgroup of $\Aut^{\I}(\kG(S))$ 
and the absolute Galois group $G_\Q$ embeds in $\Out^{\I}(\kG(S))$, we have (cf.\ \cite[Proposition~4, (ii)]{LS}, for a similar result for 
the profinite Grothendieck-Teichm\"uller group $\GT$):

\begin{corollary}\label{selfnorm}For a connected hyperbolic surface $S$ such that $d(S)>1$, the subgroup $\Inn(\kG^{\pm}(S))$ 
is its own normalizer in $\Aut^{\I}(\kG(S))$. In particular, the image of an element of $G_\Q$ 
corresponding to complex conjugation is self-centralizing in $\Out^{\I}(\kG(S))$.
\end{corollary}

The analogy between the second item of Theorem~\ref{autpants} and the corresponding isomorphisms in the topological case
$\Inn(\G^{\pm}(S))=\Aut(\G^{\pm}(S))\cong\Aut(\PG^{\pm}(S))$, for $d(S)>1$ and $S\neq S_{1,2},S_2$,
falls short only in that we do not know whether there holds $\Aut^{\I}(\kG^{\pm}(S))=\Aut(\kG^{\pm}(S))$. 
However, thanks to a recent result by Hoshi, Minamide and Mochizuki (cf.\ \cite{HMM}), we are able to fill this gap in genus $0$. 
Since, in this case, we also know that the congruence subgroup property holds (i.e.\ $\hG^{\pm}(S)=\kG^{\pm}(S)$), we get:

\begin{theorem}\label{genus0}For $g(S)=0$ and $n(S)\geq 5$, there holds:
\[\Inn(\hG^{\pm}(S))=\Aut(\hG^{\pm}(S))=\Aut(\hPG^{\pm}(S)).\]
\end{theorem}

\begin{remark}Theorem~\ref{genus0} should be compared with \cite[Corollary C]{HMM} which states that the group $\Out(\hPG(S))$ is isomorphic 
to the direct product of the profinite Grothendieck-Teichm\"uller group with the group of permutations of the punctures 
(cf.\ also \cite{MN}, for the case of profinite braid groups).
\end{remark}

A group $G$ is \emph{complete} if its center is trivial and all its automorphisms are inner, so that there is a natural isomorphism 
$G\cong\Aut(G)$. Theorem~\ref{genus0} then provides, to our knowledge, the first example of a finitely generated, infinite, residually finite, 
complete group whose profinite completion is also complete.

Let us recall that the \'etale fundamental group of an algebraic variety $X$ over $\Spec\R$ fits into the short exact sequence:
\begin{equation}\label{fundexsequence}
1\to\pi_1^\mathrm{et}(X_\C)\to\pi_1^\mathrm{et}(X)\to G_\R\to 1,
\end{equation}
where we let $X_\C:=X\times\Spec\C$ and omit base points, while $G_\R\cong\Z/2$ is the absolute Galois group of the reals, generated by 
complex conjugation.

The short exact sequence~\eqref{fundexsequence} determines a representation $\rho_\R\co G_\R\to\Out(\pi_1^\mathrm{et}(X_\C))$ and we let
$\Out_{G_\R}(\pi_1^\mathrm{et}(X_\C))$ be the centralizer of the image of $\rho_\R$ in $\Out(\pi_1^\mathrm{et}(X_\C))$.
Let also $\Aut_\R(X)$ be the group of automorphisms of $X$ defined over $\Spec\R$.
We say that $X$ satisfies the Aut-form of Grothendieck \emph{real} anabelian conjecture if there is a natural isomorphism:
\[\Aut_\R(X)\cong\Out_{G_\R}(\pi_1^\mathrm{et}(X_\C)).\]

Let us assume that the center of the group $\pi_1^\mathrm{et}(X_\C)$ is trivial and that the short exact sequence~\eqref{fundexsequence}
identifies the latter group with a characteristic subgroup of $\pi_1^\mathrm{et}(X)$. 
From \cite[Corollary~1.5.7 and Lemma~1.6.2 (or, better, its proof)]{N}, it then follows that there is a natural isomorphism:
\[\Out_{G_\R}(\pi_1^\mathrm{et}(X_\C))\cong\Out(\pi_1^\mathrm{et}(X)).\]
Thus, in this case, the Aut-form of Grothendieck real anabelian conjecture can also be formulated, in an \emph{absolute} form, 
by saying that there is a natural isomorphism:
\[\Aut_\R(X)\cong\Out(\pi_1^\mathrm{et}(X)).\]

\begin{remark}For more details on the Grothendieck anabelian conjectures and results in this direction, 
we refer to the survey \cite{MNT}, treating the case of curves, and to \cite{HM,Moch2} for the case 
of configuration spaces on curves. 
\end{remark}

Theorem~\ref{genus0} can then be rephrased as a real anabelian property for the moduli spaces $\cM_{0,n}$ of $n$-pointed, genus  zero curves.
Let $(\cM_{0,n})_\R:=\cM_{0,n}\times\Spec\R$ and let $\pi_1^\mathrm{et}((\cM_{0,n})_\R)$ be its \'etale fundamental group for some 
choice of geometric base point. Let us denote by $\Sigma_n$ the symmetric group on $n$ letters. We have:

\begin{corollary}\label{absanabelian}For $n\geq 5$, there is a natural isomorphism:
\[\Aut_\R((\cM_{0,n})_\R)\cong\Out(\pi_1^\mathrm{et}((\cM_{0,n})_\R))\cong\Sigma_n.\]
\end{corollary}
\medskip

\begin{remark}The rational version of the absolute anabelian property for $\cM_{0,n}$
was proved by Ihara and Nakamura in \cite{IN} (cf.\ Corollary~C ibid.).
\end{remark}

A few words about the proof of Theorem~\ref{autpants}. As we already observed above, the isomorphism 
$\Inn(\kG^{\pm}(S))\cong\Aut(\kC_P(S))$ is a substantial refinement of 
\cite[Theorem~8.1]{BF}. It is obtained by showing that there is indeed a coherent and symmetric way to define an orientation on the 
procongruence pants complex $\kC_P(S)$ (cf.\ Section~\ref{orientpants}). The proof of the identity $\Inn(\kG^{\pm}(S))=\Aut^{\I}(\kPG^{\pm}(S))$ 
is then based on a further improvement of the above isomorphism. An element $f$ of $\Aut^{\I}(\kPG(S))$ acts on the vertex set of $\kC_P(S)$, 
since this is naturally identified with the profinite set $\{\hI_\s|\, \s\in\kC(S)_{d(S)-1}\}$. 
Theorem~\ref{mainlemma} states that $f$ is induced by an inner automorphism 
of $\kG^{\pm}(S)$ as soon as this action sends the vertex set of an edge of $\kC_P(S)$ to the vertex set of another edge. Hence, 
in order to prove the identities of items (ii) and (iii) of Theorem~\ref{autpants}, we have to show that the elements of 
$\Aut^{\I}(\kPG(S))$ induced by the restriction of elements of $\Aut^{\I}(\kPG^{\pm}(S))$ have this property. 
By an induction argument, we are able to reduce to the case $g(S)=0$.

We then proceed by considering the action of antiholomorphic involutions of $\kPG^{\pm}(S)$ on the procongruence curve complex $\kC(S)$.
By Lemma~\ref{preserve}, for $g(S)=0$, there is only one $\kG^{\pm}(S)$-conjugacy class of antiholomorphic involutions in $\kPG(S)$. Thus,
after composing with an inner automorphism of $\kG^{\pm}(S)$, we can assume that a given automorphism of $\kPG^{\pm}(S)$ 
fixes an antiholomorphic involution and so preserves its fixed-point locus in $\kC(S)$. 

By Lemma~\ref{fixed} and Proposition~\ref{boundaryfixedpro}, 
such a fixed-point locus is finite and consists of isotopy classes of simple closed curves on $S$ which have between them geometric intersection 
either $0$ or $2$. Since pairs of curves with geometric intersection $2$ correspond to edges of the pants complex, we can apply 
Theorem~\ref{mainlemma} and conclude that the given automorphism is in fact induced by an inner automorphism of $\kG^{\pm}(S)$.
\medskip

\noindent
{\bf Acknowledgements}. We thank Javier Aramayona and Dan Margalit for some useful comments on a preliminary version of this manuscript
which helped us to spot an error in the proof and improve the main result of the paper. We thank the referees for their careful reading and 
for their thorough comments, which led us to add many important details and improve the presentation.

\section{Two preliminary results}\label{preliminaryresults}
The following two sections will be devoted to improving one of the main results of \cite{BF} (cf.\ \cite[Theorem~8.1]{BF} 
and Theorem~\ref{completepantsrigidity}). For the convenience of the reader, we will briefly recall some notations and 
basic constructions from \cite{BF} and refer to that paper for more details. 

\subsection{Augmented Teichm\"uller spaces and procongruence moduli stacks of curves}\label{augmented}
Let $\cT(S)$\index{$\cT(S)$, the Teichm\"uller space of $S$} be the Teichm\"uller space associated to the surface $S$ and 
$\ol{\cT}(S)$\index{$\ol{\cT}(S)$, the augmented Teichm\"uller space of $S$} be the \emph{augmented Teichm\"uller space} (in \cite[Section~6.2]{BF}, we called it the Bers bordification of $\cT(S)$).
The latter can be defined as the completion of the Teichm\"uller space $\cT(S)$ with respect to the Weil--Petersson metric (cf.\ \cite[Theorem~4]{Wolpert}).

The augmented Teichm\"uller space $\ol{\cT}(S)$ is a partial $\G^\pm(S)$-equivariant compactification of $\cT(S)$ such that the quotient 
$\ol{\cT}(S)/\G(S)$ is isomorphic to the coarse moduli space $\ol{M}(S)$\index{$\ol{M}(S)$, the coarse moduli space of the Deligne--Mumford compactification  $\bM(S)$} of the DM compactification 
$\bM(S)$\index{$\bM(S)$, the Deligne--Mumford compactification of $\cM(S)$} of $\cM(S)$.

For $S$ a disconnected hyperbolic surface, let $\cT(S)$ (resp.\ $\ol{\cT}(S)$) be the direct product of the Teichm\"uller spaces (resp.\ augmented
Teichm\"uller spaces) associated to the connected components of the surface $S$. The closed strata of codimension $k+1$ in $\ol{\cT}(S)$ 
of the boundary $\partial\ol{\cT}(S):=\ol{\cT}(S)\ssm\cT(S)$ are then parameterized by the $k$-simplices of $C(S)$, and, for $\s\in C(S)_k$, 
there is a natural isomorphism $\partial\ol{\cT}(S)_\s\cong\ol{\cT}(S\ssm\s)$, where we denote by $\partial\ol{\cT}(S)_\s$ the closed stratum
associated to $\s$.

Let us denote by $\wt{\cF}(S)$ the $1$-dimensional stratum of the boundary $\partial\ol{\cT}(S)$. Then, each irreducible component of $\wt{\cF}(S)$
is isomorphic to the cuspidal bordification $\ol{\H}=\H\cup\P^1(\Q)$ of the hyperbolic plane $\H$ and $C_P(S)$ identifies with the $1$-skeleton of a 
$\G(S)$-equivariant triangulation of $\wt{\cF}(S)$.

Let $\{\cM(S)^\l\}_{\l\in\Lambda}$  and $\{\bM(S)^\l\}_{\l\in\Lambda}$ be, respectively, the inverse systems of all congruence level structures 
over $\cM(S)$ and of their compactifications over $\bM(S)$ (cf.\ \cite[Section~2.5 and Section~4.1]{BF}).
Let then $\M(S):=\varprojlim_{\l\in\Lambda}\cM(S)^\l$ \index{$\M(S)$, the inverse limit of congruence level structures over $\cM(S)$} and 
$\ol{\M}(S):=\varprojlim_{\l\in\Lambda}\bM(S)^\l$ \index{$\ol{\M}(S)$, the inverse limit of congruence level structures over $\bM(S)$} be their respective inverse limits. 
There is a natural action of $\kG^{\pm}(S)$ on $\M(S)$ and $\ol{\M}(S)$ and a natural $\G^{\pm}(S)$-equivariant embedding $\ol{\cT}(S)\hookra\ol{\M}(S)$
with dense image, where $\G^{\pm}(S)$ acts on $\ol{\M}(S)$ via the natural monomorphism $\G^{\pm}(S)\hookra\kG^{\pm}(S)$.

For $\G^\l$ a level of $\G(S)$ contained in an abelian level $\G(m)$ for some $m\geq 3$, there is a natural
isomorphism $\bM(S)^\l\cong\ol{\cT}(S)/\G^\l$. Let us denote by $\cF^\l(S)$ the $1$-dimensional stratum of the DM boundary $\partial\bM(S)^\l$.
Then, the quotient $C_P^\l(S):=C_P(S)/\G^\l$ identifies with the $1$-skeleton of a $\G(S)/\G^\l$-equivariant triangulation of $\cF^\l(S)$.
Therefore, the inverse limit $\kC_P(S):=\varprojlim_\l C_P^\l(S)$ of all such finite quotients, taken in the category of 
simplicial profinite complexes, identifies with the $1$-skeleton of a $\kG(S)$-equivariant 
triangulation of the $1$-dimensional stratum $\F(S)=\varprojlim_\l\cF^\l(S)$ of the DM boundary $\partial\ol{\M}(S):=\ol{\M}(S)\ssm\M(S)$
(cf.\  \cite[Proposition~8.3]{BF} and preceding discussion).

\subsection{Automorphisms of the procongruence pants complex and orientations}\label{orientpants}
For $d(S)=1$, the pants complex $C_P(S)$ coincides with the Farey graph $\rm F$.\index{$\rm F$, the Farey graph} This is the $1$-skeleton of a $2$-dimensional simplicial complex
$\Delta$, whose geometric realization $|\Delta|$ is identified with a tessellation of the cuspidal bordification $\ol{\H}$ of the hyperbolic plane. 
The homeomorphism $|\Delta|\cong\ol{\H}$ becomes a conformal isomorphism when $|\Delta|$ is given the piecewise equilateral flat structure.
The \emph{orientation of the Farey graph} $\rm F$ (and hence of $C_P(S)$ for $d(S)=1$) is then simply the orientation of $\Delta$ associated to
the complex structure of $\ol{\H}$. 

For $\G^\l$ a finite index subgroup of $\Aut^+(\rm F)\cong\PSL_2(\Z)$ contained in an abelian level $\G(m)$, for $m\geq 2$, 
the quotient $\rm F^\l:=\rm F/\G^\l$ is the $1$-skeleton of the triangulation $\Delta^\l:=\Delta/\G^\l$ of the closed Riemann surface $\ol{\H}/\G^\l$
induced by the Farey triangulation on $\ol{\H}$. The \emph{orientation of the quotient graph} $\rm F^\l$ is then the orientation 
of the triangulation $\Delta^\l$ induced by the complex structure of $\ol{\H}/\G^\l$. For $\G^{\l'}\leq\G^\l$, the induced map $\rm F^{\l'}\to\rm F^\l$
respects orientations, so that we obtain an orientation on the inverse limit $\hF:=\varprojlim_\l\rm F^\l$ of these finite quotients 
(the \emph{profinite Farey graph}, which, again, we realize as a $1$-dimensional simplicial profinite complex) and, in particular, 
for the procongruence pants complex $\kC_P(S)\cong\hF$, for  $d(S)=1$. 

For a multicurve $\s$, let $C_P(S\ssm\s)$ be the disjoint union of the pants complexes associated to the connected components of $S\ssm\s$ and
let $d(S\ssm\s)$ be the sum of the modular dimensions of the connected components of $S\ssm\s$.
For $\s\in C(S)_{d(S)-2}$, we have that $d(S\ssm\s)=1$, so that the pants complex $C_P(S\ssm\s)$ is isomorphic to the Farey graph $\rm F$ and
identifies with a subgraph of $C_P(S)$, which we denote by $\rm F_\s$. The pants complex $C_P(S)$ is then the infinite union 
of the Farey subgraphs $\{\rm F_\s\}_{\s\in C(S)_{d(S)-2}}$ and we give each of them the orientation defined above.
A corollary of Margalit's series of isomorphisms~\eqref{Margalitiso} is then that automorphisms of $C_P(S)$ either preserve or reverse
all orientations of the Farey subgraphs.

The procongruence pants complex $\kC_P(S)$, for  $d(S)>1$ is the union of the profinite set of profinite Farey graphs 
$\{\hF_\s\}_{\s\in\kC(S)_{d(S)-2}}$, each naturally associated to a $(d(S)-2)$-simplex of the procongruence 
curve complex $\kC(S)$ (cf.\ \cite[Definition~6.3]{BF}), and we give each of them the orientation defined above.
However, it is not clear whether the automorphisms of $\kC_P(S)$ act in synchrony on the orientations of its profinite Farey subgraphs.

Let us denote by $O(S)$ the finite set of the topological types of $(d(S)-1)$-multicurves on $S$. To remedy the above issue, in \cite{BF}, we 
associated a character $\chi_\s\co\Aut(\kC_P(S))\to\{\pm 1\}$ to each $\s\in O(S)$ in the following way. The tautological action of $\Aut(\kC_P(S))$ 
on $\kC_P(S)$ preserves profinite Farey subgraphs and $\kG(S)$-orbits of profinite Farey subgraphs. We assigned to an automorphism 
$\phi\in\Aut(\kC_P(S))$ the plus or minus sign according to whether $\phi$ sends or not the fixed orientation of $\hF_\s$ to the orientation 
of $\phi(\hF_\s)$. We then proved that there is an exact sequence, for $d(S)>1$ and $S\neq S_{1,2}$ (cf.\ \cite[Theorem~8.1]{BF}):
\begin{equation}\label{originalexact}
1\to\Inn(\kG(S))\to \Aut(\kC_P(S))\to \prod_{O(S)} \{\pm 1\},
\end{equation}
where, for $S= S_{1,2}$, the group $\Aut(\kC_P(S_{1,2}))$ has to be replaced with the subgroup of 
those automorphisms preserving the set of separating curves.

\begin{remark}\label{nontrivial}It is easy to check that all characters $\chi_\s\co\Aut(\kC_P(S))\to\{\pm 1\}$, for $\s\in O(S)$, are nontrivial.
For this, let us observe that the action of an element $f\in\G^{\pm}(S)\ssm\G(S)$ on the augmented Teichm\"uller space $\ol{\cT}(S)$ 
inverts the orientation of every open stratum of this space. This implies that, for all $\s\in O(S)$, we have $\chi_\s(\inn f)=-1$.
\end{remark}

The first result of this section is an improved version of \cite[Theorem~8.1]{BF}. We will show that all the above characters 
are in fact synchronized and that the case $S= S_{1,2}$ is not exceptional, so that we have:

\begin{theorem}\label{completepantsrigidity}
For a connected hyperbolic surface $S$ such that $d(S)>1$, there is a short exact sequence:
\[1\to\Inn(\kG(S))\to \Aut(\kC_P(S))\to\{\pm 1\}\to 1.\] 
\end{theorem}

\subsection{Proof of Theorem~\ref{completepantsrigidity} for $S=S_{0,5}$}\label{05case}
There is only one topological type of $0$-simplices in $C(S_{0,5})$. Hence, the exact sequence~\eqref{originalexact} takes the simple form:
\[1\to\Inn(\kG(S_{0,5}))\to \Aut(\kC_P(S_{0,5}))\to\{\pm 1\}.\]

From Remark~\ref{nontrivial}, it then immediately follows that the above sequence is also right exact, which proves Theorem~\ref{completepantsrigidity} 
for $S=S_{0,5}$. 

\subsection{Proof of Theorem~\ref{completepantsrigidity} for $S=S_{1,2}$}\label{orientS_{1,2}}
To deal with this case, we first need  to analyze in more detail the case $S=S_{0,5}$.
Let $\partial\ol{\M}(S_{0,5})$ be the DM boundary of the inverse limit $\ol{\M}(S_{0,5})$. As explained in Section~\ref{augmented},
the profinite pants complex $\kC_P(S_{0,5})$ identifies with the $1$-skeleton of a triangulation of $\partial\ol{\M}(S_{0,5})$.
Let $\Aut(\partial\ol{\M}(S_{0,5}))$ be the group of automorphisms which restrict to a conformal or an anticonformal map on each
irreducible component and $\Aut^+(\partial\ol{\M}(S_{0,5}))$ its subgroup consisting of conformal automorphisms. 

\begin{lemma}\label{aut+}The natural action of $\kG(S_{0,5})$ on $\partial\ol{\M}(S_{0,5})$ induces an isomorphism 
$\kG(S_{0,5})\cong\Aut^+(\partial\ol{\M}(S_{0,5}))$.
\end{lemma}

\begin{proof}For simplicity, put $S=S_{0,5}$ and, for levels $\G^\l\subseteq\G^\mu\subseteq\G(S)$, let us denote by 
$\hat{\pi}_\l\co\partial\ol{\M}(S)\to\partial\ol{\cM}(S)^\l$ 
and $\pi_{\l\mu}\co\partial\ol{\cM}(S)^\l\to\partial\ol{\cM}(S)^\mu$ the natural projections. Then, 
a conformal automorphism $\phi$ of $\partial\ol{\M}(S)$ determines, and is determined by, 
the directed inverse system of conformal maps constructed as follows.  
For every $\l\in\L$, there is a $\mu\in\L$ such that the composition $\phi_\l:=\hat{\pi}_\l\circ\phi\co\partial\ol{\M}(S)\to\partial\ol{\cM}(S)^\l$ 
factors through a conformal map $\phi_{\mu\l}\co\ol{\cM}(S)^\mu\to\ol{\cM}(S)^\l$. 
It is clear that, if $\phi_{\mu'\l'}\co\ol{\cM}(S)^{\mu'}\to\ol{\cM}(S)^{\l'}$ is another such map with $\G^{\l'}\subseteq\G^\l$
and $\G^{\mu'}\subseteq\G^\mu$, we have $\pi_{\l'\l}\circ\phi_{\mu'\l'}=\phi_{\mu\l}\circ\pi_{\mu'\mu}$. 
Therefore, the set of conformal maps $\{\phi_{\mu\l}\}_{\l,\mu\in\L}$ is a directed inverse system with inverse limit the map $\phi$.

Let us now observe that, for $S=S_{0,5}$, the DM boundary $\dd\ol{\cM}(S)^\l$ of the level structure $\ol{\cM}(S)^\l$ coincides with
the Fulton curve $\cF^\l\subset\ol{\cM}(S)^\l$ (cf.\ \cite[Definition~6.1]{BF}). The conclusion then follows from \cite[Lemma~8.8]{BF}.
\end{proof}

\begin{lemma}\label{orientation05}$\Aut^+(\partial\ol{\M}(S_{0,5}))$ is an index $2$ subgroup of $\Aut(\partial\ol{\M}(S_{0,5}))$, that is to say,
every automorphism of $\partial\ol{\M}(S_{0,5})$ either preserves or reverses the orientation of all irreducible components simultaneously. 
In particular, there is a natural isomorphism $\kG^\pm(S_{0,5})\cong\Aut(\partial\ol{\M}(S_{0,5}))$.
\end{lemma}

\begin{proof}The irreducible components of $\partial\ol{\M}(S_{0,5})$ are parameterized by the $0$-simplices in $\kC(S_{0,5})$.
Let us denote by $\partial\ol{\M}(S_{0,5})_\s$ the irreducible component associated to the $0$-simplex $\s\in\kC(S_{0,5})_0$.
Let us then define the character $\chi_\s\co\Aut(\partial\ol{\M}(S_{0,5}))\to\{\pm 1\}$
which takes the value $+1$ on $f\in\Aut(\partial\ol{\M}(S_{0,5}))$ if $f$ sends the standard orientation of $\partial\ol{\M}(S_{0,5})_\s$ 
to the standard orientation of $f(\partial\ol{\M}(S_{0,5})_\s)$ and $-1$ otherwise. There is an exact sequence:
\[1\to\Aut^+(\partial\ol{\M}(S_{0,5}))\to\Aut(\partial\ol{\M}(S_{0,5}))\to\prod_{\s\in\kC(S_{0,5})_0}\{\pm 1\}.\]
In particular, $\Aut^+(\partial\ol{\M}(S_{0,5}))$ is a normal subgroup of $\Aut(\partial\ol{\M}(S_{0,5}))$.

We can now argue exactly as in \cite[Section~8.7]{BF} and conclude that the representation 
$\Aut(\partial\ol{\M}(S_{0,5}))\to\prod_{\s\in\kC(S_{0,5})_0}\{\pm 1\}$ is constant on the $\kG(S_{0,5})$-orbit of 
$0$-simplices of $\kC(S_{0,5})$. Since there is only one such orbit, the first statement of the lemma follows.
The second statement then follows from Lemma~\ref{aut+}.
\end{proof}

\begin{lemma}\label{autprocurve}There is a natural isomorphism $\Aut(\kC_P(S_{0,5}))\cong\Aut(\partial\ol{\M}(S_{0,5}))$.
\end{lemma}

\begin{proof}As we remarked above, $\kC_P(S_{0,5})$ identifies with the $1$-skeleton of a triangulation of $\partial\ol{\M}(S_{0,5})$.
By \cite[Theorem~6.7 and Proposition~8.3]{BF}, there is then a natural monomorphism 
$\Aut(\kC_P(S_{0,5}))\hookra\Aut(\partial\ol{\M}(S_{0,5}))$. Thus, the conclusion follows from Lemma~\ref{orientation05} 
and the fact that $\kG^\pm(S_{0,5})$ identifies with a subgroup of $\Aut(\kC_P(S_{0,5}))$.
\end{proof}

For $[C,P_1,P_2]\in\cM_{1,[2]}$, there is a unique elliptic involution $\u$ on $C$ and, if we denote by $C_\u$ the (genus $0$) quotient of $C$ 
by this involution and by $B_\u$ the branch locus of the orbit map $C\to C_\u$, then, the assignment $[C,P_1,P_2]\mapsto[C_\u,Q,B_\u]$,
where $Q$ is the image of the pair of points $P_1,P_2$ in $C_\u$, defines a morphism of DM stacks $\cM_{1,[2]}\to\cM_{0,1[4]}$, where
we denote by $\cM_{0,1[4]}$ the moduli stack of genus $0$ projective smooth curves with $5$ labeled points, one of which is singled out
and the others are left unordered. The morphism $\cM_{1,[2]}\to\cM_{0,1[4]}$ is a $\Z/2$-gerbe which is split, since the composition
$\cM_{1,2}\to\cM_{1,[2]}\to\cM_{0,1[4]}$ is an isomorphism of DM stacks. This can be checked observing that the morphism 
$\cM_{1,2}\to\cM_{0,1[4]}$ induces an isomorphism both on the respective coarse moduli spaces and on the isotropy groups of points.

The $\Z/2$-gerbe $\cM_{1,[2]}\to\cM_{0,1[4]}$ then induces on topological fundamental groups a split short exact sequence:
\begin{equation}\label{split1}
1\to\Z/2\to\pi_1(\cM_{1,[2]})\to\pi_1(\cM_{0,1[4]})\to 1,
\end{equation}
which, in terms of mapping class groups, can be described as follows. 

Let $\u\in\G(S_{1,2})$ be the hyperelliptic involution, let $S_{/\u}$ be the quotient of the surface $S_{1,2}$ by $\u$ and let $B_\u$ 
be the branch locus of the orbit map $S\to S_{/\u}$. The surface $S_{/\u}$ is a $1$-punctured sphere, there is a diffeomorphism 
$S_{/\u}\ssm B_\u\cong S_{0,5}$ and, if we denote by $Q$ the puncture of $S_{0,5}$ which corresponds to the puncture of $S_{/\u}$ 
via the above diffeomorphism, by Birman--Hilden theory, there is a split short exact sequence (cf.\ \cite[Theorem~2.3]{BM}, for instance):
\begin{equation}\label{split2} 
1 \to \langle\u\rangle \to \G(S_{1,2}) \to \G(S_{0,5})_Q\to 1,
\end{equation}
where $\G(S_{0,5})_Q$ is the stabilizer of the puncture $Q$ in $\G(S_{0,5})$ and the splitting is given by internal direct product
decomposition $\G(S_{1,2})=\langle\u\rangle\cdot\PG(S_{1,2})$. There is then a natural isomorphism between the split short exact 
sequences~\eqref{split1} and~\eqref{split2}. 

In particular, the epimorphism $\G(S_{1,2}) \to \G(S_{0,5})_Q$ induces an isomorphism $\PG(S_{1,2})\cong\G(S_{0,5})_Q$ 
which identifies $\PG(S_{0,5})$ with a subgroup of $\PG(S_{1,2})$. We record the following for future use:

\begin{lemma}\label{squares}The group $\PG(S_{0,5})$ identifies with the normal subgroup of $\PG(S_{1,2})$ generated by 
squares of nonseparating Dehn twists.
\end{lemma}

\begin{proof}The image of a nonseparating Dehn twist 
via the epimorphism $\G(S_{1,2}) \to \G(S_{0,5})_Q$ is a braid twist. Hence, the image of the square of a nonseparating Dehn twist
is a Dehn twist about a simple closed curve on $S_{0,5}$ which bounds a $2$-punctured disc. The subgroup of $\G(S_{1,2})$ generated
by squares of nonseparating Dehn twists is contained in $\PG(S_{1,2})$ and has trivial intersection with $\langle\u\rangle$. Therefore, it 
identifies with the subgroup of $\G(S_{0,5})_Q$ generated by Dehn twists about simple closed curves on $S_{0,5}$ bounding $2$-punctured 
discs, which is $\PG(S_{0,5})$.
\end{proof}

From the isomorphism $\cM_{1,2}\cong\cM_{0,1[4]}$, it also follows that there is a natural finite \'etale morphism $\cM_{0,5}\to\cM_{1,2}$, 
so that $\cM_{0,5}$ identifies with a level structure over $\cM_{1,[2]}$. Moreover, the DM compactification $\bM_{0,5}$ of $\cM_{0,5}$ coincides 
with the DM compactification of the latter as a level structure. Therefore, there are natural isomorphisms
$\M(S_{0,5})\cong\M(S_{1,2})$, $\ol{\M}(S_{0,5})\cong\ol{\M}(S_{1,2})$ and then $\partial\ol{\M}(S_{0,5})\cong\partial\ol{\M}(S_{1,2})$.

The Teichm\"uller spaces $\cT(S_{0,5})$ and $\cT(S_{1,2})$ are the universal covers of $\cM_{0,5}$ and $\cM_{1,2}$, respectively. Hence,
there is a natural $\G(S_{1,2})$-equivariant isomorphism $\cT(S_{0,5})\cong\cT(S_{1,2})$. 
There is also a compatible $\G(S_{1,2})$-equivariant isomorphism of curve complexes $C(S_{0,5})\cong C(S_{1,2})$. Since the 
Weil--Petersson metric on the Teichm\"uller space is determined, modulo strong equivalence, by the Fenchel--Nielsen coordinates, 
it follows that the isomorphism $\cT(S_{0,5})\cong\cT(S_{1,2})$ induces a 
natural $\G(S_{1,2})$-equivariant isomorphism of augmented Teichm\"uller spaces $\ol{\cT}(S_{0,5})\cong\ol{\cT}(S_{1,2})$ and then 
$\partial\ol{\cT}(S_{0,5})\cong\partial\ol{\cT}(S_{1,2})$. 

Even though the pants graph $C_P(S_{1,2})$ is the $1$-skeleton of a $\G(S_{1,2})$-equivariant triangulation of $\dd\ol{\cT}(S_{1,2})$,
the same holds for $C_P(S_{0,5})$ and $\partial\ol{\cT}(S_{0,5})$, and the vertex sets of the pants graphs $C_P(S_{1,2})$ and $C_P(S_{0,5})$
naturally identify, there is no natural map between these two graphs to account for it. This follows from a general result by Aramayona 
(cf.\ \cite[Theorem~A]{Aramayona}) but also from a careful analysis of the pairs of curves with minimal intersection occurring in $S_{1,2}$ and $S_{0,5}$.

In any case, from the natural $\G(S_{1,2})$-equivariant bijective map of $0$-simplices $C_P(S_{1,2})_0\cong C_P(S_{0,5})_0$, passing to 
$\kG(S_{1,2})$-completions, we get a natural $\kG(S_{1,2})$-equivariant bijective map $\kC_P(S_{1,2})_0\cong\kC_P(S_{0,5})_0$.
By \cite[Proposition~8.3]{BF}, there is a natural monomorphism $\Aut(\kC_P(S_{1,2}))\hookra\Aut(\partial\ol{\M}(S_{1,2}))$
and so $\Aut(\kC_P(S_{1,2}))\hookra\Aut(\partial\ol{\M}(S_{0,5}))$. Lemma~\ref{autprocurve} then implies that every continuous
automorphism of $\kC_P(S_{1,2})$ induces one of $\kC_P(S_{0,5})$ (compatible on vertex sets with the $\kG(S_{1,2})$-equivariant bijection given above). 

For a given $f\in\Aut(\kC_P(S_{1,2}))$, let us denote by $\tilde f$ the induced automorphism of $\kC_P(S_{0,5})$. By Lemma~\ref{orientation05}, 
$\Aut(C_P(S_{0,5}))$ identifies with a dense subgroup of the profinite group $\Aut(\kC_P(S_{0,5}))$ and $\Inn(\kG(S_{1,2}))$ with an open subgroup 
of the same group. Therefore, after possibly composing $\tilde f$ with an element of $\Inn(\kG(S_{1,2}))$, we can assume that $\tilde f\in\Aut(C_P(S_{0,5}))$. 
In particular, the given $f$ preserves the vertex set $C_P(S_{1,2})_0\subset\kC_P(S_{1,2})_0$. 

\begin{lemma}\label{vertices}If the vertices of an edge $e$ of the procongruence pants graph $\kC_P(S)$ belong to $C_P(S)_0\subset\kC_P(S)_0$,
then $e\in C_P(S)_1\subset\kC_P(S)_1$.
\end{lemma}

\begin{proof}Let $\{\s_0,\s_1\}\subset C(S)_{d(S)-1}\subset\kC(S)_{d(S)-1}$ be the vertex set of $e$.
By the first item of \cite[Lemma~6.6]{BF}, the edge $e$ is then contained in the profinite Farey subgraph $\hF_{\s_0\cap\s_1}$, which is
obtained as the $\wh{\PSL_2(\Z)}$-completion of the Farey subgraph ${\rm F}_{\s_0\cap\s_1}\subset C_P(S)$.
Hence it is enough to prove the statement of the lemma for the profinite Farey graph $\hF$.

Given two vertices $v_0,v_1\in\rm F_0$, by basic plane hyperbolic geometry, we know that there is a unique geodesic $\g$ in $\ol{\H}$ 
connecting these two points and a finite index subgroup $\G^\l$ of $\PSL_2(\Z)$ such that, for all $\G^{\l'}\leq\G^\l$, the image of $\g$ 
in the quotient surface $\ol{\H}/\G^{\l'}$ is a simple geodesic arc.
This implies that, if $\{v_0,v_1\}$ is the vertex set of an edge of $\hF$, the distance between $v_0$ and $v_1$ in $\ol{\H}$ is $1$, which is possible
only if $\{v_0,v_1\}$ is the vertex set of an edge of $\rm F$.
\end{proof}

From Lemma~\ref{vertices}, it follows that $f$ induces an automorphism of the pants complex $C_P(S_{1,2})$.
By \cite[Theorem~1 and Theorem~2]{Margalit}, we then have that $f\in\Inn(\G^\pm(S_{1,2}))$ which completes 
the proof of the case $S=S_{1,2}$ of Theorem~\ref{completepantsrigidity}.

\subsection{Three lemmas}\label{2lemmas}The following definition will play a fundamental role in the proof, by induction, of the
general case of Theorem~\ref{completepantsrigidity}:

\begin{definition}\label{link}For $d(S)>1$, every $(d(S)-2)$-multicurve on $S$ contains at least a simple closed curve which is either 
nonseparating or bounds a $2$-punctured disc. For a fixed such simple closed curve $\g$, we then let 
$L_\g$\index{$L_\g$, the union of all profinite Farey subgraphs  $\hF_\s$ such that $\g\in\s$} be the closed subgraph of $\kC_P(S)$ 
which is the union of all profinite Farey subgraphs $\hF_\s$ such that $\g\in\s$. 
\end{definition}

Let us denote by $S_\g$ either $S\ssm\g$, for $\g$ nonseparating, or the connected component of $S\ssm\g$
of positive modular dimension, for $\g$ bounding a $2$-punctured disc. We then have:

\begin{lemma}\label{linkpants}The profinite subgraph $L_\g$ of $\kC_P(S)$ is naturally isomorphic to the procongruence pants complex $\kC_P(S_\g)$.
\end{lemma}

\begin{proof}By \cite[Remark~4.7]{Boggi}, the link 
$\Link(\g)\subset\kC(S)$ \index{$\Link(\g)$, the link of $\g$ in $\kC(S)$} is naturally isomorphic to $\kC(S_\g)$. 
Let $\xi\co\kC(S_\g)\stackrel{\sim}{\to}\Link(\g)$ be such an isomorphism. 
We can then identify the vertex set of $\kC_P(S_\g)$ with a subset of the vertex set of $\kC_P(S)$ by sending
a $(d(S_\g)-1)$-simplex $\s$ of $\kC(S_\g)$ to the $(d(S)-1)$-simplex $\xi(\s)\cup\{\g\}$ of $\Star(\g)\subset\kC(S)$.\index{$\Star(\g)$, the star of $\g$ in $\kC(S)$}  
The image of this map is precisely the vertex set of the subgraph $L_\g$ of $\kC_P(S)$. In order to prove that it also induces a map between the edges of 
$\kC_P(S_\g)$ and those of $L_\g$, let us observe the following. Let $\mathring{L}_\g\subset C_P(S)$ be 
the union of all Farey subgraphs $\operatorname{F}_\s$ such that $\g\in\s$. Then, $\mathring{L}_\g$ is a dense subgraph of $L_\g$ 
and it is clear that the map defined above maps the edge set of the pants complex $C_P(S_\g)$ onto the edge set of $\mathring{L}_\g$.
Since the map defined above is $\kG(S_\g)$-equivariant, this implies that it maps the edge set of $\kC_P(S_\g)$ onto the edge set of $L_\g$.
\end{proof}

\begin{remark}\label{groupthreal}Let $\cG(\hG(S))$
\index{$\cG(\kG(S))$, the set of closed subgroups of $\kG(S)$} 
be the profinite set of closed subgroups of $\kG(S)$. By \cite[Theorem~6.9]{Boggi}, for every $0\leq i\leq d(S)-1$,
the assignment $\s\mapsto \hI_\s$, for $\s\in\kC(S)_i$, defines a $\kG(S)$-equivariant continuous embedding:
\[\cI_i\co\kC(S)_i\hookra\cG(\kG(S)).\]
The procongruence curve complex $\kC(S)$ then identifies with the abstract simplicial 
profinite complex $\kC_\cI(S)$
\index{$\kC_\cI(S)$ abstract simplicial profinite complex whose set of $i$-simplices is the set of closed subgroups $\{\hI_\s\}_{\s\in\kC(S)_i}$}
whose set of $h$-simplices is the set of of closed subgroups $\{\hI_\s\}_{\s\in\kC_h(S)}$.
\end{remark}

By Remark~\ref{groupthreal}, there is a natural continuous action of $\Aut^\I(\kPG(S))$ on $\kC(S)$. We have: 

\begin{lemma}\label{faithful}For $d(S)\geq 1$, there is a natural continuous monomorphism: 
\[\check{\Theta}_\I\co\Aut^\I(\kPG(S))\hookra\Aut(\kC(S)).\]
\end{lemma}

\begin{proof}By \cite[Theorem~7.3]{BF}, the kernel of the homomorphism $\check{\Theta}_\I$ is contained in the image of the monomorphism
$\exp\co\Hom(\kPG(S)/Z(\kPG(S)),Z(\kPG(S)))\to\Aut(\kPG(S))$, where we denote by $Z(\kPG(S))$ the center of $\kPG(S)$. 
This implies the lemma for $S\neq S_{1,1}, S_2$, since in this case $Z(\kPG(S))=\{1\}$. Otherwise, the center
$Z(\kPG(S))$ is generated by the hyperelliptic involution $\u$ and we have 
$\Hom(\kPG(S)/\langle\u\rangle,\langle\u\rangle)\cong\Hom(\PG(S)/\langle\u\rangle,\langle\u\rangle)\cong\Z/2$.

The automorphism $\exp\phi\in\Aut(\kPG(S))$ in the image of $0\neq\phi\in\Hom(\kPG(S)/\langle\u\rangle,\langle\u\rangle)$ (cf.\ \cite[Lemma~7.4]{BF})
maps a nonseparating Dehn twist $\tau_\g\in\PG(S)$ to $\u\cdot\tau_\g$. Therefore, $\exp\phi\notin\Aut^\I(\kPG(S))$, which implies the lemma.
\end{proof}

There is a natural action on the link $\Link(\g)\cong\kC(S_\g)$ of the stabilizer $\Aut^\I(\kPG(S))_{\hI_\g}$ 
of the inertia group $\hI_\g$ associated to $\g$ in $\Aut^\I(\kPG(S))$. We have:

\begin{lemma}\label{factor}After identifying $\Link(\g)$ with $\kC(S_\g)$, the action of $\Aut^\I(\kPG(S))_{\hI_\g}$ on $\Link(\g)$ 
factors through the natural action of $\Aut^\I(\kPG(S_\g))$ on $\kC(S_\g)$. The same statement holds after replacing $\kPG(S)$ with 
$\kPG^{\pm}(S)$ and $\kPG(S_\g)$ with $\kPG^{\pm}(S_\g)$.
\end{lemma}

\begin{proof}An automorphism of $\kPG(S)$ which preserves the procyclic inertia group $\hI_\g$ also preserves its centralizer $Z_{\kPG(S)}(\hI_\g)$
in $\kPG(S)$ and, since, for $\s\in\Link(\g)$, the inertia group $\hI_\s$ is contained in $Z_{\kPG(S)}(\hI_\g)$, the action of $\Aut^\I(\kPG(S))_{\hI_\g}$ 
on $\Link(\g)$ factors through the homomorphism induced by restriction:
 \[\Aut^\I(\kPG(S))_{\hI_\g}\to\Aut^\I(Z_{\kPG(S)}(\hI_\g)).\]

By \cite[Corollary~6.1]{Boggi}, there is a natural isomorphism $Z_{\kPG(S)}(\hI_\g)\cong\kPG(S)_{\g}$, so that we can identify 
$\Aut^\I(Z_{\kPG(S)}(\hI_\g))$ with the closed subgroup $\Aut^\I(\kPG(S)_{\g})$ of $\Aut(\kPG(S)_{\g})$ consisting of those automorphisms 
which preserve the set of conjugacy classes of the procyclic subgroups of $\kPG(S)_{\g}$ generated by profinite Dehn twists. 

Since $\g$ is either nonseparating or bounding a 2-punctured disc and  the pure mapping class group of 
$3$-punctured sphere is trivial, by \cite[Theorem~4.10]{BF}, we have the exact sequences:
\[1\ra\kPG(S)_{\vec{\g}}\to\kPG(S)_{\g}\to\{\pm 1\}\hspace{0.5cm}\mbox{ and }\hspace{0.5cm}
1\ra\hat{\mathrm{I}}_\g\to\kPG(S)_{\vec{\g}}\to\kPG(S_\g)\to 1,\]
where the homomorphism $\kPG(S)_{\g}\to\{\pm 1\}$ is induced by the action of the stabilizer $\PG(S)_\g$ on the orientation of the
simple closed curve $\g$. In particular, for $\g$ bounding a $2$-punctured disc, we have $\kPG(S)_{\vec{\g}}=\kPG(S)_{\g}$.

Since $\kPG(S)_{\vec{\g}}$ is the normal subgroup of $\kPG(S)_{\g}$ topologically generated by Dehn twists, an element of $\Aut^\I(\kPG(S)_{\g})$
preserves the subgroup $\kPG(S)_{\vec{\g}}$, so that there is a natural homomorphism:
\[\Aut^\I(\kPG(S)_{\g})\to\Aut^\I(\kPG(S)_{\vec\g}).\] 

By \cite[Theorem~4.14]{BF}, the procyclic subgroup $\hI_\g$ is the center of $\kPG(S)_{\vec\g}$, hence, it is preserved by every element of 
$\Aut^\I(\kPG(S)_{\vec\g})$ and there is a natural homomorphism:
\[\Aut^\I(\kPG(S)_{\vec\g})\to\Aut^\I(\kPG(S_\g)).\] 

By composing all the above homomorphisms, we get a natural homomorphism:
\[\Aut^\I(\kPG(S))_{\hI_\g}\to\Aut^\I(\kPG(S_\g)).\]  

The natural isomorphism $\Link(\g)\cong\kC(S_\g)$ identifies all the inertia subgroups of $\kPG(S)$ contained in $Z_{\kPG(S)}(\hI_\g)\cong\kPG(S)_{\g}$, 
but which do not contain the Dehn twist $\tau_\g$, with the inertia subgroups of $\kPG(S_\g)$ in a way which is clearly compatible with the above series of
homomorphisms. The first statement of the lemma follows. The second can be proved in a similar way.
\end{proof}

\subsection{Connectedness of various curve complexes}
For the proof of Theorem~\ref{completepantsrigidity}, we need the connectivity of some curve complexes.
Even though these results are probably well known to experts, we include a proof for lack of suitable references.

\begin{definition}\label{newcurvecompl}\leavevmode
\begin{enumerate}
\item For $n\geq 4$, let $C_b(S_{0,n})$ be the curve complex defined as the full subcomplex of the complex of curves $C(S_{0,n})$
whose vertices consist of isotopy classes of simple closed curves on $S_{0,n}$ which bound a $2$-punctured disc.
\item For $n\geq 2$, let $C_{0b}(S_{1,n})$ be the full subcomplex of the curve complex $C(S_{1,n})$ whose vertices are isotopy classes 
of either nonseparating simple closed curves or simple closed curves which bound a $2$-punctured disc on $S_{0,n}$.
\end{enumerate}
\end{definition}

\begin{lemma}\label{complexconnected}\leavevmode
\begin{enumerate}
\item For $n\geq 5$, the simplicial complex $C_b(S_{0,n})$ is connected.
\item For $n\geq 2$, the simplicial complex $C_{0b}(S_{1,n})$ is connected.
\end{enumerate}
\end{lemma}
\begin{proof}(i): This can be proved by the same argument which proves the connectivity of the standard curve complex 
(cf., for instance, the proof of \cite[Theorem~4.3]{FM}). We use induction on the geometric intersection number of two simple closed curves 
$a$ and $b$ which bound a $2$-punctured disc on $S_{0,n}$. When they are disjoint there is nothing to prove. Let us then assume that $a$ and $b$ 
have geometric intersection $i(a,b) >0$. We claim that there is a simple closed curve $c$ (bounding a $2$-punctured disc) such that 
$i(a,c)=0$ and $i(c,b) < i(a,b)$. 

First, we construct a simple closed curve $c'$ (not necessarily bounding a $2$-punctured disc), such that $i(a,c')=0$ and $i(c',b) < i(a,b)$, 
by  following the oriented path $a$ until it intersects $b$, 
further following $b$ until it intersects again $a$ for the first time and eventually continuing along the path $a$ in order to close the loop. 
There are several possibilities depending on the orientations of $a$ and $b$, but, in the end, we find an essential simple closed curve with 
the required properties. 

Now, $c'$ bounds a $k$-punctured disc $D$, where $k\geq 2$, which does not contain $a$. If $k=2$, then we take $c=c'$ and we are done.
Otherwise, we take for $c$ the boundary of a $2$-punctured disc $D'$ contained in $D$ such that $c$ is the union of an arc contained in $c'$ and 
an arc which is either part of $b\cap D$ or is disjoint from $b$. In both cases, we have that $i(c,b)\leq i(c',b)< i(a,b)$ and, obviously, $i(a,c)=0$.
\smallskip

\noindent
(ii): We proceed as above.  Here, the curve $c$ is either nonseparating or 
it bounds a $2$-punctured disc but the proof above works without any essential change. 
\end{proof}

\subsection{Proof of Theorem~\ref{completepantsrigidity} for $g(S)=0$}\label{Propgenus0}
We proceed by induction on $n\geq 5$. The case $n=5$ was proved in Section~\ref{05case} and serves as base for the induction.
Let us then assume that Theorem~\ref{completepantsrigidity} holds for $S_{0,n-1}$ and let us prove it for $S_{0,n}$.

For $n\geq 5$, every simplex $\s\in O(S_{0,n})$ contains at least a simple closed curve $\g$ on $S_{0,n}$ which bounds a $2$-punctured disc. 
We then have:

\begin{lemma}\label{orientlink}If an automorphism $\phi\in\Aut(\kC_P(S_{0,n}))$ preserves the orientation of the profinite Farey
subgraph $\hF_\s$ of $\kC_P(S_{0,n})$ (cf.\ Section~\ref{orientpants}), then it preserves the orientations of all profinite subgraphs 
$\hF_{\s'}$ such that $\g\in\s'$, where $\g\in\s$ is a simple closed curve which bounds a $2$-punctured disc.
\end{lemma}

\begin{proof}By \cite[Theorem~6.7]{BF}, for all hyperbolic surfaces $S$, there is a natural monomorphism: 
\[\check{\Theta}_P\co\Aut(\kC_P(S))\hookra\Aut(\kC(S)),\]  
which is induced by the identification of the vertices of $\kC_P(S)$ with the facets of $\kC(S)$. 

Let us briefly recall the proof
of this theorem. For $d(S)=0$, we have that $\kC_P(S)_0=\kC(S)_0$ and $\dim\kC(S)=0$ and so the statement is obvious. 
For $d(S)>1$, from  \cite[Lemma~6.6, (ii)]{BF}), it follows that the continuous automorphisms of $\kC_P(S)$ preserve
Farey subgraphs and so there is a continuous action of $\Aut(\kC_P(S))$ on the profinite set of Farey subgraphs of $\kC_P(S)$.
These are parameterized by the profinite set of $(d(s)-2)$-simplices of $\kC(S)$, thus there is an induced continuous action on the latter set.
By \cite[Lemma~6.6, (iii)]{BF}), this continuous action induces a monomorphism $\Aut(\kC_P(S))\hookra\Aut(\kC^\ast(S))$,
where $\kC^\ast(S)$ is the dual graph of $\kC(S)$ (cf.\cite[Definition~3.9]{BF}). By \cite[Lemma~6.5]{BF},
we then have that $\Aut(\kC^\ast(S))=\Aut(\kC(S))$ and the conclusion follows.

Thus, after possibly composing the given automorphism $\phi\in\Aut(\kC_P(S_{0,n}))$ with an element in the image of $\Inn(\kG(S_{0,n}))$, 
we can assume that its image $\check{\Theta}_P(\phi)$ in $\Aut(\kC(S_{0,n}))$ preserves the $0$-simplex $\{\g\}\in\kC(S_{0,n})$ and so $\phi$ 
preserves the subgraph $L_\g$ of $\kC_P(S_{0,n})$. Since, by Lemma~\ref{linkpants}, we have that $L_\g\cong\kC_P(S_\g)$, from 
the induction hypothesis, it follows that, if the automorphism $\phi$ preserves the orientation of some profinite Farey subgraph of 
$L_\g$, then $\phi$ preserves the orientation of all profinite Farey subgraphs $\hF_\s$ such that $\g\in\s$. 
\end{proof}

By (i) of Lemma~\ref{complexconnected}, the curve complex $C_b(S_{0,n})$ is connected. Thus, there is a set $\g_1,\dots,\g_k$ of simple closed curves 
on $S_{0,n}$ bounding a $2$-punctured disc such that any two representatives $\s$ and $\s'$
of the set of orbits $O(S_{0,n})$ are contained in a chain $\Link(\g_1),\dots,\Link(\g_k)$ with the property that the intersection $L_{\g_i}\cap L_{\g_{i+1}}$, 
for $1\leq i\leq k-1$, contains at least a profinite Farey subgraph. Lemma~\ref{orientlink} and a simple induction then imply that an automorphism 
$\phi\in\Aut(\kC_P(S_{0,n}))$, which preserves the orientation of $\hF_\s$, also preserves 
the orientation of $\hF_{\s'}$. This completes the proof of Theorem~\ref{completepantsrigidity} for $g(S)=0$.

\subsection{Proof of Theorem~\ref{completepantsrigidity} for $g(S)\geq 1$}
Let us first consider the case $g(S)=1$. Here, we need to use the curve complex $C_{0b}(S_{1,n})$ instead of the curve complex $C_b(S_{0,n})$ 
and (ii) instead of (i) of Lemma~\ref{complexconnected}. We then proceed by induction on $n\geq 2$. 

The base of the induction is provided by the case $S=S_{1,2}$ proved above. 
The induction step essentially proceeds as in the case $g(S)=0$ (cf.\ Section~\ref{Propgenus0}). 
The only difference is that, if $\g_i$, for $1\leq i\leq k$, is a nonseparating simple closed curve on $S$, then the link $\Link(\g_i)$
is isomorphic to the procongruence curve complex of a genus $0$ surface, so that, in this case, instead of the induction hypothesis, 
we need to use the genus $0$ case of Theorem~\ref{completepantsrigidity}, which we already proved.

For $g(S)\geq 2$, we proceed by induction on the genus where the base of the induction is the genus $1$ case proved above.
The relevant curve complex here is the complex of nonseparating curves $C_0(S)$, which, for $g(S)\geq 2$, (cf.\ \cite[Theorem~4.4]{FM}) 
is connected. The rest of the argument proceeds as in the previous cases.

\subsection{A rigidity criterion}
From Theorem~\ref{completepantsrigidity}, it follows that, for $d(S)>1$, there is a natural isomorphism:
\begin{equation}\label{basiciso}
\Inn(\kG^{\pm}(S))\cong\Aut(\kC_P(S)).
\end{equation}
From this isomorphism, we will derive a characterization of those elements of $\Aut^\I(\kPG(S))$ which are induced by
an inner automorphism of $\kG^{\pm}(S)$. Before we state the result, we need to make the following remark:

\begin{remark}\label{gtrealization}
In the group-theoretic realization of the procongruence curve complex $\kC(S)$ which we described in Remark~\ref{groupthreal}, 
the vertices of the procongruence pants complex $\kC_P(S)$ are identified with the set $\{\hI_\s\}_{\s\in\kC(S)_{d(S)-1}}$ of inertia groups 
of $\kPG(S)$ of maximal rank. The natural faithful continuous action of $\Aut^\I(\kPG(S))$ on the curve complex $\kC(S)$ (cf.\ Lemma~\ref{faithful}) 
then induces a continuous faithful action of $\Aut^\I(\kPG(S))$ on the vertex set $\kC_P(S)_0$ of the procongruence pants complex. 
\end{remark}

We have:

\begin{theorem}\label{mainlemma}Let $S$ be a connected hyperbolic surface such that $d(S)> 1$. An element $f\in\Aut^\I(\kPG(S))$ 
is in the image of $\Inn(\kG^{\pm}(S))\to\Aut^\I(\kPG(S))$ if and only if, for some edge $\{v_0,v_1\}\in \kC_P(S)_1$, the set of vertices $\{f(v_0),f(v_1)\}$ 
is also an edge of $\kC_P(S)$. 
\end{theorem}

For the proof, we will need the following simple lemma in group theory:

\begin{lemma}\label{Wells}Let $1\to H\to G\to L\to 1$ be a short exact sequence of groups and let $f$ be an automorphism 
of $H$ and suppose that:
\begin{enumerate}
\item the center of $H$ is trivial;
\item the image $\bar f$ of $f$ in $\Out(H)$ normalizes the image of the outer representation $\rho\co L\to\Out(H)$ 
associated to the given short exact sequence;
\item the automorphism of $\rho(L)$ induced by the restriction of $\inn\bar f$ lifts to an automorphism of $L$.
\end{enumerate}
Then, $f$ extends to an automorphism of $G$.
\end{lemma}

\begin{proof}For an element $f\in\Aut(H)$, we denote by $\bar f$ its image in $\Out(H)$.
Let then $\mathrm{Comp}(L,H)$ be the subgroup of $\Aut(L)\times\Aut(H)$ formed by the pairs $(\psi,f)$
such that, for all $\alpha\in L$, there holds (in $\Out(H)$):
\[\bar f\rho(\alpha)\bar f^{-1}=\rho(\psi(\alpha)).\]

Since, by hypothesis (i), the center of $H$ is trivial, according to Wells' exact sequence (cf.\ \cite[Theorem]{Wells}), there is a canonical isomorphism:
\[\Aut(G)_H\cong\mathrm{Comp}(L,H),\]
where $\Aut(G)_H$ is the subgroup of $\Aut(G)$ consisting of those automorphisms which preserve $H$.
This isomorphism sends an element $\tilde f\in\Aut(G)_H$ to the pair $(\psi,f)$, where $\psi\in\Aut(L)$ is the automorphism 
induced by $\tilde f$ passing to the quotient by the normal subgroup $H$ and $f$ is the restriction of $\tilde f$ to $H$.

The conclusion follows if we show that an $f\in\Aut(H)$, which satisfies the hypotheses (ii) and (iii) of the lemma, is part of a compatible pair. 
Since $\inn\bar f$ preserves the subgroup $\rho(L)$ and the induced automorphism lifts to $\phi\in\Aut(L)$, it is clear that $(\phi,f)$ 
is such a compatible pair.
\end{proof}

\subsection{Proof of Theorem~\ref{mainlemma} for $S=S_{0,5}$}\label{mainlemma=1} 
Let $f\in\Aut^\I(\kPG(S))$ be an element satisfying the hypotheses of the theorem. 
Since, for $S=S_{0,5}$, there is only one topological type of both vertices and edges of the pants complex $C_P(S)$,
the action of $\kG(S)$ on the set of oriented edges of $\kC_P(S)$ is transitive. Therefore, 
for any edge $e=\{\alpha_0,\alpha_1\}$ of $\kC_P(S)$, there is an element $x_e\in\kG(S)$ such that $\alpha_i=x_e\cdot v_i\cdot x_e^{-1}$, for $i=0,1$
(cf.\ Remark~\ref{gtrealization}).

Let us consider the short exact sequence $1\to\kPG(S)\to\kG(S)\to\Sigma_n\to 1$, where we put $n:=n(S)$. As it is well known, 
$\cM_{0,3}$ is a point and, for $n\geq 4$, the moduli space $\cM_{0,n}$ is naturally isomorphic to the configuration space of $n-3$ ordered points on 
$\P^1\ssm\{0,1,\infty\}$, since $\Aut(\P^1)$ acts simply transitively on triples of distinct points of $\P^1$. Hence,
the group which in \cite[Corollary~C]{HMM} is denoted by $\Pi_n$ is isomorphic to $\kPG(S)$, so that \cite[Corollary~C]{HMM} implies 
that the outer action of $\Sigma_n$ on $\kPG(S)$, associated to the short exact sequence above, identifies $\Sigma_n$ with a normal subgroup 
of $\Out(\kPG(S))$. 

In particular, the automorphism $f\in\Aut^\I(\kPG(S))$ satisfies the second hypothesis of Lemma~\ref{Wells}. Since $\kPG(S)$ is center-free 
and the outer representation $\Sigma_n\to\Out(\kPG(S))$ is faithful, the other two hypotheses are also satisfied. 
Therefore, by Lemma~\ref{Wells}, $f$ extends to an automorphism $\tilde f$ of $\kG(S)$
which preserves procyclic inertia groups (cf.\ Definition~\ref{inertiagroups}) and so is contained in $\Aut^\I(\kG(S))$ (cf.\ Definition~\ref{inertiapreserving}).

There is then a series of identities, for $i=0,1$:
\[f(\alpha_i)=\tilde f(\alpha_i)=\tilde f(x_e\cdot v_i\cdot x_e^{-1})=\tilde f(x_e)\cdot \tilde f(v_i)\cdot \tilde f(x_e)^{-1}=\tilde f(x_e)\cdot f(v_i)\cdot\tilde f(x_e)^{-1},\]
so that $\{f(\alpha_0),f(\alpha_1)\}=\inn \tilde f(x_e)(\{f(v_0),f(v_1)\})\in\kC_P(S)_1$.
Therefore, the continuous action of the automorphism $f$ on the profinite set of vertices of $\kC_P(S)$ extends to a continuous action 
on the procongruence pants complex $\kC_P(S)$. The conclusion then follows from the isomorphism~\eqref{basiciso}.

\subsection{Proof of Theorem~\ref{mainlemma} for $S=S_{1,2}$}\label{mainlemma=2} 
As for the proof of the case $S=S_{1,2}$ of Theorem~\ref{completepantsrigidity} (cf.\ Section~\ref{orientS_{1,2}}), the idea is to reduce to the
genus $0$ case. However, as in the proof of Theorem~\ref{completepantsrigidity}, this is complicated by the fact that, although 
the curve complexes $C(S_{1,2})$ and $C(S_{0,5})$ are naturally isomorphic, there is not even a map between the corresponding pants graphs.
To deal with this case, we will then need first to prove a series of lemmas:

\begin{lemma}\label{S_{1,2}}For $S=S_{1,2}$, the action of $\Aut^\I(\kPG(S_{1,2}))$ on $\kC(S_{1,2})$ preserves topological types. Moreover, there is
a natural monomorphism $\Aut^\I(\kPG(S_{1,2}))\hookra\Aut^\I(\kPG(S_{0,5}))$ induced by restriction of automorphisms.
\end{lemma}

\begin{proof}With the notation of Section~\ref{2lemmas}, for $\g$ a simple closed curve on $S_{1,2}$, by \cite[Corollary~4.12]{BF}, 
there is a natural isomorphism $Z_{\kPG(S_{1,2})}(\hI_\g)\cong\kPG(S_{1,2})_{\g}$ and, by \cite[Theorem~4.10]{BF}, there are exact sequences:
\[1\ra\kPG(S_{1,2})_{\vec{\g}}\to\kPG(S_{1,2})_{\g}\to\{\pm 1\}\hspace{0.5cm}\mbox{ and }\hspace{0.5cm}
1\ra\hat{\mathrm{I}}_\g\to\kPG(S_{1,2})_{\vec{\g}}\to\kPG((S_{1,2})_\g)\to 1.\]

After possibly composing with an inner automorphism, we can assume that a given element $f\in\Aut^\I(\kPG(S_{1,2}))$ is such that $\g':=f(\g)$ 
also belongs to $C(S_{1,2})_0\subset\kC(S_{1,2})_0$. Since $\kPG(S_{1,2})_{\vec{\g}}$ identifies with the subgroup of the centralizer 
$Z_{\kPG(S_{1,2})}(\hI_\g)$ topologically generated by profinite Dehn twists, we have that $f(\kPG(S_{1,2})_{\vec{\g}})=\kPG(S_{1,2})_{\vec{\g}'}$. 
By \cite[Theorem~4.14]{BF}, the procyclic subgroup $\hI_\g$ is the center of $\kPG(S_{1,2})_{\vec{\g}}$, so that $f$ induces an isomorphism 
$\bar f\co\kPG((S_{1,2})_\g)\stackrel{\sim}{\to}\kPG((S_{1,2})_{\g'})$.

For $\g$ separating, we have that $\kPG((S_{1,2})_\g)\cong\wh{\SL(2,\Z)}$ while, for $\g$ nonseparating, we have that $\kPG((S_{1,2})_\g)$
is a free group in two generators. The latter profinite group is torsion-free while the former is not. Thus, $\g'$ has the same topological type of $\g$.

This proves the first part of the lemma. Let us then observe that, by Lemma~\ref{squares}, $\kPG(S_{0,5})$ 
identifies with the normal subgroup of $\kPG(S_{1,2})$ topologically generated by 
squares of nonseparating Dehn twists. By the previous part of the proof, elements of $\Aut^\I(\kPG(S_{1,2}))$ preserve 
this subgroup and so there is a homomorphism as claimed in the lemma. 

In order to prove that this homomorphism is injective, let us observe that the natural monomorphism $\Aut^\I(\kPG(S_{1,2}))\hookra\Aut(\kC(S_{1,2}))$ 
(cf.\ Lemma~\ref{faithful}) factors through the given homomorphism $\Aut^\I(\kPG(S_{1,2}))\to\Aut^\I(\kPG(S_{0,5}))$ and the natural
monomorphism $\Aut^\I(\kPG(S_{0,5}))\hookra\Aut(\kC(S_{0,5}))$, via the isomorphism $\kC(S_{1,2})\cong\kC(S_{0,5})$.
This concludes the proof of the lemma. 
\end{proof}

Since all the groups involved are center-free, for $n\geq 4$, there is a series of natural isomorphisms: 
\begin{equation}\label{sigmaiso}
\Inn(\kG(S_{0,n}))\left/\Inn(\kPG(S_{0,n}))\right.\cong\kG(S_{0,n})\left/\kPG(S_{0,n})\right.\cong\Sigma_n.
\end{equation}

Let us denote by $\Out^\I_{\Sigma_n}(\kPG(S_{0,n}))$\index{$\Out^\I_{\Sigma_n}(\kPG(S_{0,n}))$, the centralizer of $\Sigma_n$ in $\Out^\I(\kPG(S_{0,n}))$} the centralizer of the image of $\Sigma_n$ in
$\Out^\I(\kPG(S_{0,n}))$. Let us recall that $\Out^\sharp(\kPG(S_{0,n}))$\index{$\Out^\sharp(\kPG(S_{0,n}))$, the subgroup of elements of $\Out(\kPG(S_{0,n}))$ commuting with $\Sigma_n$
and preserving each conjugacy class of the procyclic subgroups of $\kPG(S_{0,n})$ generated by a Dehn twist about a
simple closed curve bounding a $2$-punctured disc}  was defined in \cite[Section~0.1]{HS}, for $n\geq 4$,
to be the subgroup of $\Out(\kPG(S_{0,n}))$ consisting of those elements which commute with the image of $\Sigma_n$
and preserve each conjugacy class of the procyclic subgroups of $\kPG(S_{0,n})$ generated by a Dehn twist about a
simple closed curve bounding a $2$-punctured disc in $S_{0,n}$. We have:

\begin{lemma}\label{HarbSch}For $n=4$ and $5$, there holds $\Out^\sharp(\kPG(S_{0,n}))=\Out^\I_{\Sigma_n}(\kPG(S_{0,n}))$.
\end{lemma}

\begin{proof}For $n=4,5$, all essential simple closed curves on $S_{0,n}$ bound a $2$-punctured disc. Therefore, there is at least an inclusion
$\Out^\sharp(\kPG(S_{0,n}))\subseteq\Out^\I_{\Sigma_n}(\kPG(S_{0,n}))$. To prove that the reverse inclusion holds, we have to show
that the elements of $\Out^\I_{\Sigma_n}(\kPG(S_{0,n}))$ preserve every conjugacy class of procyclic subgroup of $\kPG(S_{0,n})$ 
generated by a Dehn twist. By definition, the group $\Out^\I(\kPG(S_{0,n}))$ acts on the set of such conjugacy classes with its subgroup
$\Sigma_n$ acting transitively on it. This easily implies that those elements which commute with the action of $\Sigma_n$
act trivially on this set.
\end{proof}

From the isomorphism $\PG(S_{1,2})\cong\G(S_{0,5})_Q$ (cf.\ Section~\ref{orientS_{1,2}}) and the isomorphisms~\eqref{sigmaiso}, 
it follows that there is a series of isomorphisms $\Inn(\PG(S_{1,2}))\left/\Inn(\kPG(S_{0,5}))\right.\cong(\Sigma_5)_Q\cong\Sigma_4$. 
Let us then define:
\begin{itemize}
\item $\wt{\Out}^\I(\kPG(S_{1,2})):=\Aut^\I(\kPG(S_{1,2}))\left/\Inn(\kPG(S_{0,5}))\right.$;
\index{$\wt{\Out}^\I(\kPG(S_{1,2})):=\Aut^\I(\kPG(S_{1,2}))\left/\Inn(\kPG(S_{0,5}))\right.$}
\item $\wt{\Out}^\I_{\Sg_4}(\kPG(S_{1,2}))$
\index{$\wt{\Out}^\I_{\Sg_4}(\kPG(S_{1,2}))$, the centralizer of the image of $\Sigma_4$ in $\wt{\Out}^\I(\kPG(S_{1,2}))$} 
to be the centralizer of the image of $\Sigma_4$ in $\wt{\Out}^\I(\kPG(S_{1,2}))$;
\item $\Aut^\I_{\Sg_5}(\kPG(S_{0,5}))$
\index{$\Aut^\I_{\Sg_5}(\kPG(S_{0,5}))$,  the inverse image of $\Out^\I_{\Sg_5}(\kPG(S))$ in $\Aut(\kPG(S_{0,5}))$}
 (resp.\ $\Aut^\I_{\Sg_4}(\kPG(S_{1,2}))$)
 \index{$\Aut^\I_{\Sg_4}(\kPG(S_{1,2}))$, the inverse image of \ $\wt{\Out}^\I_{\Sg_4}(\kPG(S_{1,2}))$ in $\Aut(\kPG(S_{1,2}))$}  
 to be the inverse image of $\Out^\I_{\Sg_5}(\kPG(S))$ (resp.\ $\wt{\Out}^\I_{\Sg_4}(\kPG(S_{1,2}))$) in $\Aut(\kPG(S_{0,5}))$ 
 (resp.\ $\Aut(\kPG(S_{1,2}))$). 
\end{itemize}

Note that $\Aut^\I_{\Sg_5}(\kPG(S_{0,5}))\subset\Aut^\I(\kPG(S_{0,5}))$ and $\Aut^\I_{\Sg_4}(\kPG(S_{1,2}))\subset\Aut^\I(\kPG(S_{1,2}))$. 
We then have:

\begin{lemma}\label{GT}There are natural isomorphisms:
\[\begin{array}{l}
\wt{\Out}^\I(\kPG(S_{1,2}))\cong\Sigma_4\times\wt{\Out}^\I_{\Sg_4}(\kPG(S_{1,2})),\\
\\
\wt{\Out}^\I_{\Sg_4}(\kPG(S_{1,2}))\cong\Out^\I(\kPG(S_{1,2})),\\
\\
\Out^\I(\kPG(S_{0,5}))\cong\Sigma_5\times\Out^\I_{\Sg_5}(\kPG(S_{0,5})).
\end{array}\] 
Moreover, the natural monomorphism of Lemma~\ref{S_{1,2}} restricts to a monomorphism: 
\[\Aut^\I_{\Sg_4}(\kPG(S_{1,2}))\hookra\Aut^\I_{\Sg_5}(\kPG(S_{0,5})).\]
\end{lemma}

\begin{proof}There is a short exact sequence:
\[1\to\Sigma_4\to\Aut^\I(\kPG(S_{1,2}))\left/\Inn(\kPG(S_{0,5}))\right.\to\Out^\I(\kPG(S_{1,2}))\to 1\]

By \cite[Corollary~C]{HMM} (cf.\ also the proof of Lemma~\ref{HosMinMoch}), the image of $\Sg_5$ in $\Out(\kPG(S_{0,5}))$ 
(and so in $\Out^\I(\kPG(S_{0,5}))$) identifies with a normal subgroup. Therefore, there is also a short exact sequence:
\[1\to\Sigma_5\to\Out^\I(\kPG(S_{0,5}))\to\Out^\I(\kPG(S_{0,5}))\left/\Sigma_5\right.\to 1\]

Since $\Sigma_4$ and $\Sigma_5$ are complete groups, the above short exact sequences split (cf.\ \cite[Theorem~7.15]{Rotman}) 
and there are natural isomorphisms as stated in the lemma. The last statement of the lemma follows from the fact that 
the centralizer of $\Sigma_4$ in $\Sigma_5$ is trivial.
\end{proof}


Let $f\in\Aut^\I(\kPG(S_{1,2}))$ be an element such that for some edge $\{v_0,v_1\}\in \kC_P(S_{1,2})_1$, the set of vertices $\{f(v_0),f(v_1)\}$ 
is also an edge of $\kC_P(S_{1,2})$. Let $\td f$ be the image of $f\in\Aut^\I(\kPG(S_{1,2}))$ via the monomorphism 
$\Aut^\I(\kPG(S_{1,2}))\hookra\Aut^\I(\kPG(S_{0,5}))$ of Lemma~\ref{S_{1,2}}. 

The element $\td f$ then acts on the vertex set $\kC_P(S_{0,5})_0$
through the element $f$ and the natural continuous $\kG(S_{1,2})$-equivariant bijection on vertex sets:
 \[q\co\kC_P(S_{1,2})_0\stackrel{\sim}{\to} \kC_P(S_{0,5})_0,\]
 so that, for $v\in\kC_P(S_{1,2})_0$, there holds $\td f(q(v))=q(f(v))$. The key lemma is the following:

\begin{lemma}\label{reductionstep}The element $\td f\in\Aut^\I(\kPG(S_{0,5}))$ is such that for some edge $\{w_0,w_1\}\in \kC_P(S_{0,5})_1$, 
the set of vertices $\{\td f(w_0),\td f(w_1)\}$ is also an edge of $\kC_P(S_{0,5})$.
\end{lemma}

\begin{proof}By Lemma~\ref{S_{1,2}}, an element $f\in\Aut^\I(\kPG(S_{1,2}))$ preserves $\kG(S_{1,2})$-orbits. Hence,
if the given $\{v_0,v_1\}\in \kC_P(S_{1,2})_1$ is an edge contained in a $\kG(S_{1,2})$-orbit such that the pair of vertices $\{q(v_0),q(v_1)\}$ 
is an edge of $\kC_P(S_{0,5})$, it follows that $\{\td f(q(v_0)),\td f(q(v_1))\}$ is also an edge of $\kC_P(S_{0,5})$. 
In this case, we just let $w_0:=q(v_0)$ and $w_1:=q(v_1)$ and we are done. 
As it is easy to check, in particular, this happens if the common profinite simple closed curve in the intersection $v_0\cap v_1$ has the 
topological type of a separating curve on $S_{1,2}$.

Let us then consider the case when the common profinite simple closed curve $\g$ in $v_0\cap v_1$ has the topological type 
of a nonseparating curve on $S_{1,2}$. As usual, it is not restrictive to assume that $\g\in C(S_{1,2})_0\subset\kC(S_{1,2})_0$. Moreover,
by the first isomorphism of Lemma~\ref{GT}, after, possibly, composing the given $f\in\Aut^\I(\kPG(S_{1,2}))$ with an inner automorphism 
of $\PG(S_{1,2})$, we can also assume that $f\in\Aut^\I_{\Sg_4}(\kPG(S_{1,2}))$. 

From the last statement of Lemma~\ref{GT} and Lemma~\ref{HarbSch}, it then follows that the conjugacy class of 
$\hI_\g\cap\kPG(S_{0,5})$ in $\kPG(S_{0,5})$ is preserved by $\td f$. Therefore, after, possibly, composing by an inner 
automorphism of $\kPG(S_{0,5})$, we can at last assume that $f\in\Aut^\I_{\Sg_4}(\kPG(S_{1,2}))_{\hI_\g}$. 

Since $f$ fixes the vertex $\g$ (cf.\  Remark~\ref{gtrealization}), it also preserves the link $\Link(\g)$ of $\g$ in $\kC(S_{1,2})$ and then the vertex set 
of the profinite subgraph $L_\g$ of $\kC_P(S_{1,2})$. Let us recall (cf.\ Lemma~\ref{linkpants}) that there are natural $\kPG(S_{1,2})_\g$-equivariant 
continuous isomorphisms $\Link(\g)\cong\kC(S_{1,2}\ssm\g)$ and $L_\g\cong\kC_P(S_{1,2}\ssm\g)$. By hypothesis, we then have that both 
$\{v_0,v_1\}$ and $\{f(v_0),f(v_1)\}\in(L_\g)_1$. We claim that this implies that $f$ preserves the edge set of $L_\g$ 
(and so induces an automorphism of this $1$-dimensional simplicial profinite complex). 

By Lemma~\ref{factor}, $f$ acts on the vertex set $(L_\g)_0\cong\kC_P(S_{1,2}\ssm\g)_0$
through its image via the natural homomorphism:
\[R_\g\co\Aut^\I(\kPG(S_{1,2}))_{\hI_\g}\to\Aut^\I(\kPG(S_{1,2}\ssm\g)),\] 
induced by the restriction to the stabilizer $\kPG(S_{1,2})_{\vec\g}$ followed by the projection to its quotient $\kPG(S_{1,2}\ssm\g)$.

 Note that $S_{1,2}\ssm\g\cong S_{0,4}$. Hence, there is a short exact sequence:
\begin{equation}\label{stabext}
1\to\kPG(S_{1,2}\ssm\g)\to\kG(S_{1,2}\ssm\g)\to\Sigma_4\to 1
\end{equation}
and so a faithful representation $\Sigma_4\hookra\Out^\I(\kPG(S_{1,2}\ssm\g))$. Let then $\Out_{\Sg_4}^\I(\kPG(S_{1,2}\ssm\g))$ be
the centralizer of the image of $\Sg_4$ in $\Out^\I(\kPG(S_{1,2}\ssm\g))$ and let $\Aut^\I_{\Sg_4}(\kPG(S_{1,2}\ssm\g))$ be the 
inverse image of $\Out_{\Sg_4}^\I(\kPG(S_{1,2}\ssm\g))$ in $\Aut^\I(\kPG(S_{1,2}\ssm\g))$.

\begin{lemma}\label{imageproj}$R_\g(\Aut^\I_{\Sg_4}(\kPG(S_{1,2}))_{\hI_\g})\subseteq\Aut^\I_{\Sg_4}(\kPG(S_{1,2}\ssm\g))$.
\end{lemma}

\begin{proof}By Lemma~\ref{GT}, by restriction, we get a natural monomorphism:
\[\Aut^\I_{\Sg_4}(\kPG(S_{1,2}))_{\hI_\g}\hookra\Aut^\I_{\Sg_5}(\kPG(S_{0,5})).\]
In order to prove the lemma, we just need to translate this statement in terms of the homomorphism $R_\g$.

Let us observe that $\kPG(S_{1,2})_{\vec\g}\cap\kPG(S_{0,5})=\kPG(S_{0,5})_{q(\g)}$ and that, after identifying $\kPG(S_{0,5}\ssm q(\g))$ 
with $\kPG(S_{0,4})$, the natural projection $\kPG(S_{0,5})_{q(\g)}\to\kPG(S_{0,5}\ssm q(\g))$ identifies
with the restriction of a forgetful homomorphism $\kPG(S_{0,5})\to\kPG(S_{0,4})$ to $\kPG(S_{0,5})_{q(\g)}$.

In \cite[Section~1.2]{HS}, the group $\Aut^\sharp(\kPG(S_{0,n}))$, for $n\geq 4$, was defined to be the inverse image of
the group $\Out^\sharp(\kPG(S_{0,n}))$ by the natural homomorphism $\Aut(\kPG(S_{0,n}))\to\Out(\kPG(S_{0,n}))$.
By Lemma~\ref{HarbSch}, we then have that $\Aut^\sharp(\kPG(S_{0,n}))=\Aut^\I_{\Sg_n}(\kPG(S_{0,n}))$, for $n=4,5$.

In \cite[Section~2.2]{HS}, it is proved that the forgetful homomorphism $\kPG(S_{0,n})\to\kPG(S_{0,n-1})$ induces a
homomorphism $\Aut^\sharp(\kPG(S_{0,n}))\to\Aut^\sharp(\kPG(S_{0,n-1}))$, for $n\geq 5$.

By Lemma~\ref{GT} and the above remarks, the restriction of the homomorphism $R_\g$ to the subgroup 
$\Aut^\I_{\Sg_4}(\kPG(S_{1,2}))_{\hI_\g}$ is then equivalent to the restriction of the homomorphism considered above
$\Aut^\sharp(\kPG(S_{0,5}))\to\Aut^\sharp(\kPG(S_{0,4}))$ (associated to the forgetful homomorphism 
$\kPG(S_{0,5})\to\kPG(S_{0,4})$) 
to the image of $\Aut^\I_{\Sg_4}(\kPG(S_{1,2}))_{\hI_\g}$ in $\Aut^\I_{\Sg_5}(\kPG(S_{0,5}))=\Aut^\sharp(\kPG(S_{0,5}))$. 
This implies the claim of the lemma.
\end{proof}

By Lemma~\ref{imageproj}, the image of $f$ in $\Out^\I(\kPG(S_{1,2}\ssm\g))$ commutes with the image of $\Sigma_4$, so that
all hypotheses of Lemma~\ref{Wells}, applied to the short exact sequence~\eqref{stabext}, are satisfied and $f$ extends to an 
automorphism of $\kG(S_{1,2}\ssm\g)$.  As in Section~\ref{mainlemma=1}, we then conclude that $f$ induces an automorphism 
of the pants complex $\kC_P(S_{1,2}\ssm\g)$, as claimed above.

It is now easy to check that, for some edge $\{v_0',v_1'\}\in(L_\g)_1\cong\kC_P(S_{1,2}\ssm\g)_1$, we have that 
$\{q(v_0'),q(v_1')\}\in\kC_P(S_{0,5})_1$ and conclude, as we did at the beginning of the proof, letting $w_0:=q(v_0')$ and $w_1:=q(v_1')$.
This concludes the proof of Lemma~\ref{reductionstep}.
\end{proof}

From Lemma~\ref{reductionstep} and the case $S=S_{0,5}$ of Theorem~\ref{mainlemma} proved above, we conclude that the image $\td f$ 
of $f$ in $\Aut^\I(\kPG(S_{0,5}))$ is in the image of $\Inn(\kG^\pm(S_{0,5}))$. In conclusion, $\td f$ is an inner automorphism of $\kG^\pm(S_{0,5})$
which normalizes its subgroup $\kPG(S_{1,2})$. 

From \cite[Lemma~9.13]{BF}, it now follows that the normalizer of $\kPG(S_{1,2})$ in
$\kG^\pm(S_{0,5})$ coincides with the closure in this group of the normalizer of $\PG(S_{1,2})$ in $\G^\pm(S_{0,5})$, which, as it easily follows
from \cite[Theorem, (ii)]{Luo}, is just $\PG^\pm(S_{1,2})$. Since $\Inn(\kG^\pm(S_{1,2}))=\Inn(\kPG^\pm(S_{1,2}))$, 
this implies Theorem~\ref{mainlemma} for $S=S_{1,2}$.

\subsection{Proof of Theorem~\ref{mainlemma} for $d(S)>2$} 
We proceed by induction on $d(S)$. Let us then assume that the statement of the lemma holds for all surfaces of modular dimension $<d(S)$.
By the isomorphism~\eqref{basiciso}, it is enough to prove that the action of the given automorphism $f\in\Aut^\I(\kPG(S))$ on the set of vertices of the 
procongruence pants complex $\kC_P(S)$ preserves its set of edges, that is to say, for every edge $\{\alpha_0,\alpha_1\}\in\kC_P(S)_1$,
there holds $\{f(\alpha_0),f(\alpha_1)\}\in\kC_P(S)_1$. 

Let us show that, by an argument similar to the proof of the case $S=S_{0,5}$ of the theorem, it is actually enough to show 
that this is the case for a set of representatives of the $\kPG(S)$-orbits in $\kC_P(S)_1$. 
Thus, assuming that $e=\{w_0,w_1\}\in\kC_P(S)_1$ is such that $f(e)=\{f(w_0),f(w_1)\}\in\kC_P(S)_1$, let us show that, if 
$e'=\{\alpha_0,\alpha_1\}\in\kC_P(S)_1$ has the property that there is an element $x\in\kG(S)$ such that 
$\alpha_i=x\cdot w_i\cdot x^{-1}$, for $i=0,1$, then we also have $f(e')=\{f(\alpha_0),f(\alpha_1)\}\in\kC_P(S)_1$.
Now, for $i=0,1$, we have the identities:
\[f(\alpha_i)=f(x\cdot w_i\cdot x^{-1})=f(x)\cdot f(w_i)\cdot f(x)^{-1},\]
so that $\{f(\alpha_0),f(\alpha_1)\}=\inn f(x)(\{f(w_0),f(w_1)\})\in\kC_P(S)_1$.

By the remark above, in particular, we can assume that $\{\alpha_0,\alpha_1\}\in C_P(S)_1\subset\kC_P(S)_1$.

For $d(S)>2$, the complexes $C_b(S)$ (for $g(S)=0$), $C_{0b}(S)$ (for $g(S)=1$) and $C_0(S)$ (for $g(S)>1$) are connected.
This implies that there is a set $\g_1,\dots,\g_k$ of simple closed curves on $S$, where $\g_i$, for $i=1,\dots,k$, is either nonseparating or
bounds a $2$-punctured disc, such that the edge $\{v_0,v_1\}$ is contained in $L_{\g_1}$, the edge $\{\alpha_0,\alpha_1\}$ is contained in $L_{\g_k}$
and the intersection $L_{\g_i}\cap L_{\g_{i+1}}$, for $1\leq i\leq k-1$, contains at least an edge of $\kC_P(S)$.

The conclusion then follows from a simple induction and the following lemma:

\begin{lemma}\label{edgepreserve}If an automorphism $f\in\Aut^\I(\kPG(S))$ sends an edge of $L_{\g_i}$ to an edge
of $\kC_P(S)$, then it sends every edge of $L_{\g_i}$ to an edge of $\kC_P(S)$, for $i=1,\dots,k$.
\end{lemma}

\begin{proof}After composing $f\in\Aut^\I(\kPG(S))$ with an element in the image of $\Inn(\kG(S))\to\Aut^\I(\kPG(S))$, we can assume that $f$ 
preserves the procyclic subgroup $\hI_{\g_i}$ and then acts on the vertex set of the subgraph $L_{\g_i}$, which identifies with the set of 
$(d(S)-1)$-simplices of the star of $\g_i$ in $\kC(S)$.
 
By Lemma~\ref{linkpants}, the profinite subgraph $L_{\g_i}$ is naturally isomorphic to $\kC_P(S_{\g_i})$ and, by Lemma~\ref{factor}, this natural 
isomorphism induces an action of $\Aut^\I(\kPG(S))_{\hI_{\g_i}}$ on the vertex set of $\kC_P(S_{\g_i})$ which factors through an element of 
$\Aut^\I(\kPG(S_{\g_i}))$. The induction hypothesis then implies that $f$ preserves the edge set of $L_{\g_i}$, for $i=1,\dots,k$.
\end{proof}

\section{Antiholomorphic involutions}\label{antiholomorphicsection}
\subsection{Centralizers of antiholomorphic involutions}
An \emph{antiholomorphic involution} $\iota\in\G^{\pm}(S)$ is an element of order $2$ (an involution) which reverses the orientation of $S$.
Any such element can be realized as the antiholomorphic involution associated to a real Riemann surface homeomorphic to $S$.

The centralizer of $\iota$ in $\G^{\pm}(S)$ has a simple description. Let $S_\iota:=S/\langle\iota\rangle$ be the quotient surface. 
Let $\Fix(\iota)$ be the fixed-point set of $\iota$. Then, $\Fix(\iota)$ is the union of a (possibly empty) set of disjoint 
simple closed curves on $S$ and the quotient surface $S_\iota$ is orientable if and only if $S\ssm\Fix(\iota)$ is not connected.
Moreover, if $\Fix(\iota)\neq\emptyset$, then $S_\iota$ is a surface with boundary $\dd S_\iota$, which coincides with the image 
of $\Fix(\iota)$ in $S_\iota$ (cf.\ \cite[Proposition~1.2]{Sch}). Let us denote by $\Map(S_\iota)$ the group of isotopy classes 
of self-diffeomorphisms of the (possibly non-orientable) surface $S_\iota$. We then have:

\begin{proposition}\label{centralizer}The centralizer $Z_{\G^{\pm}(S)}(\iota)$ of $\iota$ in $\G^{\pm}(S)$ is described by the short exact sequence:
\[1\to\langle\iota\rangle\to Z_{\G^{\pm}(S)}(\iota)\to\Map(S_\iota)\to 1.\]
\end{proposition}

\begin{proof}If $\Fix(\iota)=\emptyset$, it is enough to observe that the orientation cover $S\to S_\iota$ is canonical. 
This implies that any self-homeomorphism of $S_\iota$ lifts to $S$ and so the proposition follows in this case.

Let us then assume that $\Fix(\iota)\neq\emptyset$ and $S\ssm\Fix(\iota)$ is connected. The surface $S\ssm\Fix(\iota)$ 
identifies with the orientation cover of $S_\iota\ssm\dd S_\iota$, so that every 
self-homeomorphism of $S_\iota\ssm\dd S_\iota$ lifts to $S\ssm\Fix(\iota)$. Since every self-homeomorphism of $S$ which commutes 
with $\iota$ preserves the fixed-point set $\Fix(\iota)$, the conclusion follows.

Let us then consider the case when $S\ssm\Fix(\iota)$ is not connected. In this case, $S\ssm\Fix(\iota)$ has two connected components $S'$ and $S''$ 
such that their closures $\ol{S}'$ and $\ol{S}''$ in $S$ both identify with the quotient surface $S_\iota$. This implies that a self-homeomorphism of $S_\iota$
lifts to a pair of self-homeomorphisms of $\ol{S}'$ and $\ol{S}''$ which are compatible on the boundary and can then be glued to a 
self-homeomorphism of $S$.
\end{proof}

\subsection{The fixed-point set of an antiholomorphic involution in the augmented Teichm\"uller space}
As it is customary, for a space $X$ endowed with an involution $\iota$, we denote by $X^\iota$ its fixed-point set. 
Let then $\cT(S)$ be the Teichm\"uller space associated to the surface $S$ endowed with the Weil--Petersson metric. 
From \cite[Lemma~3.5]{MTopics}, it follows that the fixed-point set $\cT(S)^\iota$ of an antiholomorphic involution $\iota\in\G^{\pm}(S)$ 
is a nonempty and connected real analytic submanifold of $\cT(S)$ of (real) dimension $d(S)$. We have: 

\begin{lemma}\label{fixedpointsetaug}The fixed-point set $\ol{\cT}(S)^\iota$ for the action of $\iota$ on the augmented Teichm\"uller space 
$\ol{\cT}(S)$ is the closure of the fixed-point set $\cT(S)^\iota$ in $\ol{\cT}(S)$. 
\end{lemma}

\begin{proof}The augmented Teichm\"uller space $\ol{\cT}(S)$ is the completion of the Teichm\"uller space $\cT(S)$ with respect to the 
Weil--Petersson metric  (cf.\ Section~\ref{augmented}). The conclusion then follows from the fact that, for $x\in\cT(S)^\iota$ and 
$y\in\dd\ol{\cT}(S)^\iota:=\ol{\cT}(S)^\iota\ssm\cT(S)^\iota$, the unique geodesic connecting $x$ and $y$ (cf.\ the discussion 
preceding \cite[Theorem~5]{Wolpert}) is contained in $\ol{\cT}(S)^\iota$ and intersects $\cT(S)^\iota$ in an open dense subset 
(cf.\ \cite[Theorem~5]{Wolpert}).
\end{proof} 

For an antiholomorphic involution $\iota$, let us denote by $C(S)^{\iota}_k$
the set of $k$-simplices of the abstract simplicial complex $C(S)$ fixed by $\iota$.
By \cite[Theorem~5]{Wolpert} and \cite[Lemma~3.5]{MTopics}, we then have that, for $\s\in C(S)_k$, 
the fixed-point set $\partial\ol{\cT}(S)_\s^\iota$ of the corresponding closed stratum of $\partial\ol{\cT}(S)$ is nonempty 
if and only if $\s\in C(S)_k^\iota$. Moreover, by Lemma~\ref{fixedpointsetaug}, for all $\s\in C(S)_k^\iota$, there holds 
$\partial\ol{\cT}(S)_\s^\iota\cong\ol{\cT}(S\ssm\s)^\iota=\ol{\cT(S\ssm\s)^\iota}$. 
Let us sum up the above discussion in the following proposition:

\begin{proposition}\label{boundaryfixed}The closed irreducible strata of codimension $k+1$ in the boundary of the fixed-point locus $\ol{\cT}(S)^\iota$
are parameterized by the fixed-point set $C(S)_k^\iota$, for $k\geq 0$.
\end{proposition}

\begin{remark}\label{recover}Note that, for all $k\geq 0$, the action of the centralizer $Z_{\G^{\pm}(S)}(\iota)$ on $C(S)_k$ preserves the fixed-point 
set $C(S)_k^\iota$ and acts on the latter with a finite number of orbits. 
\end{remark}




\subsection{Antiholomorphic involutions of the profinite mapping class group}
Let us assume that $g(S)\leq 2$. By the congruence subgroup property in genus $\leq 2$, we then have that $\kPG(S)\cong\hPG(S)$ 
and so $\kPG(S)^\pm\cong\hPG^\pm(S)$ and $\kG^{\pm}(S)\cong\hG^{\pm}(S)$.
The augmentation map $\G^{\pm}(S)\to\Z/2$ induces an augmentation map $\hG^{\pm}(S)\to\Z/2$ and we define an 
antiholomorphic involution of $\hG^{\pm}(S)$ to be an element of order $2$ whose image by the augmentation map is nontrivial. 

\begin{proposition}\label{goodgroup}For $g(S)\leq 2$, we have:
\begin{enumerate}
\item $\G^{\pm}(S)$ is a good group, that is to say the natural homomorphism $\G^{\pm}(S)\hookra\hG^\pm(S)$ induces an isomorphism
on (continuous) cohomology with finite coefficients.
\item The natural homomorphism $\G^{\pm}(S)\hookra\hG^\pm(S)$ induces a bijection between the sets of conjugacy classes 
of involutions in the two groups. 
\item For every involution $\iota\in\G^{\pm}(S)\subset\hG^{\pm}(S)$, the centralizer $Z_{\hG^{\pm}(S)}(\iota)$ 
coincides with the closure of $Z_{\G^{\pm}(S)}(\iota)$ in $\hG^{\pm}(S)$.
\end{enumerate}
The same statement holds if we replace $\G^{\pm}(S)$ by $\PG^{\pm}(S)$.
\end{proposition}

\begin{proof}(i): It is well known and easy to prove that $\PG(S)$ is a good group for $g(S)\leq 2$ (cf.\ for instance, the closing remarks in 
\cite{Oda} or \cite[Proposition~8.5]{B1}). The conclusion then follows from the Hochschild--Lyndon--Serre spectral sequence of a group extension.
\smallskip

\noindent
(ii) and (iii): These follow from the first item and \cite[Corollary~B]{BZ}.
\end{proof}

\subsection{The center of the extended procongruence mapping class group}\label{center} 
The following result is an almost immediate consequence of \cite[Theorem~4.14]{BF}:

\begin{proposition}\label{slim}Let $U$ be an open subgroup of $\kG^\pm(S)$. 
\begin{enumerate}
\item For $S\neq S_{0,4}, S_{1,1}, S_{1,2}$ and $S_2$, the centralizer of $U$ in $\kG^\pm(S)$ is trivial.
\item For $S= S_{1,1}, S_{1,2}$ and $S_2$, the centralizer of $U$ in $\kG^\pm(S)$ is generated by the hyperelliptic involution $\u$.
\item For $S= S_{0,4}$, we have $Z(\kG^\pm(S))=\{1\}$ and  if $U\subseteq\kPG^\pm(S)$, then $Z_{\kG^\pm(S)}(U)=K_4$, where 
$K_4\cong\{\pm 1\}\times\{\pm 1\}$ is the Klein subgroup of $\G_{0,[4]}$. 
\end{enumerate}
\end{proposition}

\begin{proof}An antiholomorphic involution $\iota\in\G^\pm(S)$ acts by conjugation on Dehn twist powers $\tau_\g^k\in\PG(S)$, for $k\in\N^+$,
by the formula $\iota\tau_\g^k\iota=\tau_{\iota(\g)}^{-k}$. This implies that the outer representation associated to the short exact sequence
$1\to\kG(S)\to\kG(S)^\pm\to\Z/2\to 1$ is nontrivial when restricted to the subgroup $U$. Together with \cite[Theorem~4.14]{BF}, this implies
all claims of the proposition.
\end{proof}

\subsection{The fixed-point set of an antiholomorphic involution in the procongruence moduli stack}\label{fpset} 
The main result of this section is Proposition \ref{boundaryfixedpro}, which is a key ingredient for the proof of the  genus zero case of Theorem~\ref{autpants} 
(cf.\ Section~\ref{genus0case}). 

For every torsion-free characteristic level $\G^\l$ of $\G(S)$ and an antiholomorphic involution $\iota\in\G^{\pm}(S)$, 
let $(\cM(S)^\l,\iota)$ be the real complex manifold with equivariant 
fundamental group isomorphic to $\G^\l\cdot\langle\iota\rangle$ (cf.\ \cite[Section~3]{Huisman}). 
Let $\bM(S)^\l$ be the DM compactification of $\cM(S)^\l$. The action of $\iota$ on $\cM(S)^\l$ extends to $\bM(S)^\l$, which then acquires 
a structure of real complex analytic variety. 

Let us recall that we defined $\ol{\M}(S)$ to be the inverse limit of all compactified geometric level
structure over $\bM(S)$ and that there is a natural embedding $\ol{\cT}(S)\hookra\ol{\M}(S)$ (cf.\ Section~\ref{preliminaryresults}). 
Therefore, the action of $\iota$ on the augmented Teichm\"uller space $\ol{\cT}(S)$ extends to $\ol{\M}(S)$ (and, in particular, to $\M(S)$). 
We have:

\begin{lemma}\label{density}The fixed-point set $\ol{\M}(S)^\iota$  is the inverse limit of an inverse system of real analytic manifolds of 
dimension $d(S)$ and contains ${\M}(S)^\iota$ as an open dense subspace. 
\end{lemma}

\begin{proof}For the first claim, by the definitions involved, it is enough to prove that, for a cofinal system of congruence levels $\{\G^\l\}_{\l\in\L'}$ 
of $\G(S)$, the fixed-point set  $(\bM(S)^\l)^\iota$ is smooth of real dimension $d(S)$. From \cite[Theorem~3.11]{sym}, it follows that, at least, 
there is such a cofinal system with the property that the associated compactified level structures $\bM(S)^\l$ are smooth complex manifolds 
for all $\l\in\L'$. Since, as we observed above, the fixed-point set $\cT(S)^\iota$ is nonempty, we also know that the fixed-point set $(\bM(S)^\l)^\iota$ 
is nonempty. This implies (cf.\ \cite[Section~3.1]{MTopics}) that the fixed-point set $(\bM(S)^\l)^\iota$
is a smooth real analytic subspace of $\bM(S)^\l$ of (real) dimension $d(S)$, for all $\l\in\L'$. The claim then follows.

For the second claim, it is enough to prove that, for every characteristic level $\G^\l$ of $\G(S)$, the fixed-point set 
$(\cM(S)^\l)^\iota$ is dense in the fixed-point set  $(\bM(S)^\l)^\iota$. Every irreducible
component (as a real analytic variety) of the fixed-point set $(\bM(S)^\l)^\iota$ is the image of the fixed-point set $\ol{\cT}(S)^{\iota'}$ in the quotient 
$\bM(S)^\l=\ol{\cT}(S)/\G^\l$, for some antiholomorphic involution $\iota'$ such that $\iota'\equiv\iota\mod\G^\l$.
Hence, the conclusion follows from Lemma~\ref{fixedpointsetaug}.
\end{proof}

We then have:

\begin{proposition}\label{procentralizer}For $g(S)\leq 2$, the fixed-point locus $\ol{\M}(S)^\iota$ of an antiholomorphic involution 
$\iota\in\G^{\pm}(S)\subset\hG^{\pm}(S)$ is irreducible and contains ${\M}(S)^\iota$ as an open dense subspace. 
\end{proposition}

\begin{proof}By \cite[Theorem~3.6]{MTopics}, there is a natural bijection between the set of connected components 
of the real locus of $(\cM(S)^\l,\iota)$ and conjugacy classes of involutions in $\G^\l\cdot\langle\iota\rangle$. 

By (i) of Proposition~\ref{goodgroup} and Shapiro's lemma, $\G^\l\cdot\langle\iota\rangle$ is also good for $g(S)\leq 2$.
By \cite[Corollary~B, (i)]{BZ} and the above remarks, there is then a bijective correspondence between the set of connected 
components of the real locus of $(\cM(S)^\l,\iota)$ and conjugacy classes of involutions in $\hG^\l\cdot\langle\iota\rangle$. 
By passing to the inverse limit over all such level structures, since $\cap_{\l\in\Lambda}\hG^\l\cdot\langle\iota\rangle=\langle\iota\rangle$, 
we see that the fixed-point  locus $\M(S)^\iota$ is connected and smooth (and so also irreducible). The proposition then follows from
Lemma~\ref{density}.
\end{proof}

The closed boundary strata of $\ol{\M}(S)$ are parameterized in a $\kG^\pm(S)$-equivariant way by the simplices of the 
procongruence curve complex $\kC(S)$ and, for $\s\in C(S)\subset\kC(S)$, there holds $\dd\ol{\M}(S)_\s\cong\ol{\M}(S\ssm\s)$, 
where we denote by $\dd\ol{\M}(S)_\s$ the closed stratum of $\dd\ol{\M}(S)=\ol{\M}(S)\ssm\M(S)$ parameterized by a simplex $\s\in\kC(S)$.  

For an antiholomorphic involution $\iota$, let us denote by $\kC(S)^{\iota}_k$
the set of $k$-simplices of the simplicial profinite complex $\kC(S)$ fixed by $\iota$.
From Proposition~\ref{boundaryfixed} and Proposition~\ref{procentralizer}, it then follows:

\begin{proposition}\label{boundaryfixedpro}For $g(S)\leq 2$, the closed irreducible strata of the DM boundary of the fixed-point locus 
$\ol{\M}(S)^\iota$ of codimension $k+1$ are parameterized by the fixed-point set $\kC(S)_k^\iota$, for all $k\geq 0$. Moreover, $\kC(S)_k^\iota$ is 
the closure of the fixed-point set $C(S)_k^\iota$ in the profinite set $\kC(S)_k$.
\end{proposition}

\begin{proof}Since the closed irreducible strata of the DM boundary $\dd\ol{\M}(S)$ of codimension $k+1$ are $\kG^\pm(S)$-equivariantly 
parameterized by $\kC(S)_k$, the closed irreducible stratum $\dd\ol{\M}(S)^\iota_\s$, parameterized by $\s\in\kC(S)_k$,
is preserved by $\iota$ if and only if $\s\in\kC(S)_k^\iota$ and so, in particular, $\iota\in\kG^\pm(S)_\s$. As we observed above,
$\dd\ol{\M}(S)_\s\cong\ol{\M}(S\ssm\s)$ and $\kG^\pm(S)_\s$ acts on it through its quotient $\kG^\pm(S\ssm\s)$.
The image of $\iota$ in $\G^\pm(S\ssm\s)$ is then also a antiholomorphic involution of this group, which already implies that 
$\dd\ol{\M}(S)_\s^\iota\cong\ol{\M}(S\ssm\s)^\iota$ is nonempty. By Proposition~\ref{procentralizer}, the fixed-point locus is also irreducible.
This implies the first claim of the proposition. The last claim follows from the fact that the boundary $\dd\ol{\M}(S)^\iota=\ol{\M}(S)^\iota\ssm{\M}(S)^\iota$ 
is the closure of the boundary $\dd\ol{\cT}(S)^\iota$ in the boundary $\dd\ol{\M}(S)$, which also follows from
Proposition~\ref{procentralizer}.
\end{proof}

\subsection{Exceptional automorphisms of $\G^\pm(S)$ in low genus}\label{exAut}
\begin{lemma}\label{firstcoh}For $S=S_{1,1},S_{1,2}$ or $S_2$, there holds:
\[H^1(\G^\pm(S)/Z(\G(S)),\Z/2)=\langle\xi\rangle\oplus\langle\zeta\rangle,\] 
where the epimorphism $\zeta\co\G^\pm(S)/Z(\G(S))\to\Z/2$ is induced by the orientation character of $\G^\pm(S)$ and 
the restriction of $\xi\co\G^\pm(S)/Z(\G(S))\to\Z/2$ to $\G(S)/Z(\G(S))$ is nontrivial.
\end{lemma}

\begin{proof}The center $Z(\G(S))=Z(\G^\pm(S))$ is generated by the hyperelliptic involution $\u$ and the quotient 
$\G(S)_{/\u}:=\G(S)/Z(\G(S))$ (resp.\ $\G(S)^\pm_{/\u}:=\G(S)^\pm/Z(\G(S))$) identifies with the (resp.\ extended) mapping class group of the 
quotient surface $S/\langle\u\rangle$ marked by the set $B$ of branch points of the natural map $S\to S/\langle\u\rangle$. Note that,
for $S=S_{1,1},S_{1,2}$ and $S_2$, we have $|B|=3,4$ and $6$, respectively.
There are then natural epimorphisms $\G(S)_{/\u}\to\Sigma_B$
and $\G(S)^\pm_{/\u}\to\Sigma_B$, where $\Sigma_B$ is the symmetric group on the set of points $B$.

The inflation-restriction exact sequence associated to the natural short exact sequence
$1\to\G(S)_{/\u}\to\G(S)^\pm_{/\u}\to\Z/2\to 1$ determines the exact sequence:
\[0\to\langle\zeta\rangle\to H^1(\G(S)^\pm_{/\u}),\Z/2)\to H^1(\G(S)_{/\u},\Z/2).\]

From the results in \cite[Section~5.1.3]{FM} (essentially, because $\G(S)_{/\u}$ is normally generated by a Dehn twist), 
it follows that $H^1(\G(S)_{/\u},\Z/2)=\langle\xi\rangle$. In order to prove the lemma, we then only need to show that the cohomology
group $H^1(\G(S)^\pm_{/\u}),\Z/2)$ contains an element whose restriction to $\G(S)_{/\u}$ is nontrivial.

For this, let us observe that the epimorphisms $\G(S)^\pm_{/\u}\to\Sigma_B$ and $\G(S)_{/\u}\to\Sigma_B$ 
induce on first cohomology groups the monomorphisms  $H^1(\Sigma_B,\Z/2)\hookra H^1(\G(S)^\pm_{/\u}),\Z/2)$ and 
$H^1(\Sigma_B,\Z/2)\hookra H^1(\G(S)_{/\u}),\Z/2)$. Since $H^1(\Sigma_B,\Z/2)\cong\Z/2$, the conclusion follows.
\end{proof} 

By Proposition~\ref{slim}, there holds $Z(\hG^\pm(S))=Z(\hG(S))=Z(\G(S))=\langle\u\rangle$, for $S=S_{1,1},S_{1,2}$ or $S_2$.
Hence, with the notations of Lemma~\ref{firstcoh}, we have that, for $S=S_{1,1},S_{1,2}$ or $S_2$:
\begin{equation}\label{cohomcalculation}
H^1(\hG^\pm(S)/Z(\hG(S)),Z(\hG(S)))\cong H^1(\G^\pm(S)/Z(\G(S)),Z(\G(S)))=\langle\xi\rangle\oplus\langle\zeta\rangle,
\end{equation}
where the restriction of the cocycle $\xi$ to $\hG(S)/Z(\G(S))$ is nontrivial while the restriction of the cocycle $\zeta$ to $\hG(S)/Z(\G(S))$ is trivial.

The epimorphisms $\xi$ and $\zeta$ then determine two automorphisms $\exp\xi$ and $\exp\zeta$ of $\hG^\pm(S)$ (resp.\ $\G^\pm(S)$), 
defined by the assignments $x\mapsto x\cdot\xi(\bar x)$ and $x\mapsto x\cdot\zeta(\bar x)$, where $\bar x$ denotes the image of $x$ in the quotient 
$\hG^\pm(S)/Z(\hG(S))$ (resp.\ $\G^\pm(S)/Z(\G(S))$) (cf.\ the proof of \cite[Lemma~3.5]{BF}). Note that, by definition, both these automorphisms are not inner.

For the case $S=S_{1,2}$, we then have:

\begin{proposition}\label{(g,n)=(1,2)}For $S=S_{1,2}$, there are natural isomorphisms:
\begin{enumerate}
\item $\Aut(\hG^\pm(S))\cong\Aut(\hPG^\pm(S))\times\langle\exp\xi\rangle\times\langle\exp\zeta\rangle$;
\item $\Aut^\I(\hG^\pm(S))\cong\Aut^\I(\hPG^\pm(S))\times\langle\exp\zeta\rangle$.
\end{enumerate}
\end{proposition}

\begin{proof}We have that $\hG^{\pm}(S)=\hPG^{\pm}(S)\times\langle\u\rangle$, 
because $\u$ determines a splitting of the short exact sequence $1\to \hPG^{\pm}(S)\to \hG^{\pm}(S) \to \Z/2\to 1$.

By the Wells exact sequence associated to the extension $1\to Z(\hG(S))\to\hG(S)\to\hPG^{\pm}(S)\to 1$ and the above identity, since
$\Aut(Z(\hG(S)))=\Aut(\Z/2)=\{1\}$, there is a split short exact sequence (cf.\ \cite[Lemma~7.4]{BF}):
\[0\to H^1(\hG^\pm(S)/Z(\hG(S)),Z(\hG(S)))\to\Aut(\hG^\pm(S))\to\Aut(\hPG^\pm(S))\to 1.\]
The first item of the proposition then follows from \eqref{cohomcalculation}.

Since the quotient $\hG(S)/Z(\hG(S))$ is topologically normally generated 
by the image of a nonseparating Dehn twist $\tau_\g$ of $\G(S)$, there holds $\xi(\bar\tau_\g)=\u$, so that $\exp\xi\notin\Aut^\I(\hG^\pm(S))$.
Instead, since the restriction of $\zeta$ to $\hG(S)/Z(\hG(S))$ is trivial, we have $\exp\zeta\in\Aut^\I(\hG^\pm(S))$. 
The second item of the proposition then also follows.
\end{proof}

In the topological case, for $S=S_{1,1}, S_{1,2}$ or $S_2$, from the isomorphisms~\eqref{nonoriso1}, 
Lemma~\ref{firstcoh} and \cite[Lemma~3.5]{BF}, it follows that we have the isomorphisms
(cf.\ also \cite[Theorem~1]{McCarthy} for the case $S=S_2$):
\begin{equation}\label{nonoriso3}
\Aut(\G^\pm(S))\cong\Aut^\I(\G^\pm(S))\times\langle\exp\xi\rangle\hspace{0.5cm}\mbox{and}\hspace{0.5cm}
\Aut^\I(\G^\pm(S))=\Aut^\I(\G(S))\times\langle\exp\zeta\rangle.
\end{equation}

\section{Proof of Theorem~\ref{autpants}}
\subsection{Preliminary lemmas}
Before we proceed to the proof of Theorem~\ref{autpants}, we need to prove first a series of lemmas. The following definition will be also useful:

\begin{definition}\label{Ichar}We say that a subgroup of $\kG^{\pm}(S)$ is \emph{$\I$-characteristic} if it is preserved by all elements 
of $\Aut^\I(\kG^{\pm}(S))$.
\end{definition} 

\begin{lemma}\label{autcurvecomplex}For $S=S_{g,n}$ such that either $g\geq 1$ and $(g,n)\neq(1,2)$ or $g=0$ and $n\geq 5$, there is 
an epimorphism $\Aut^\I(\kG^\pm(S))\to\Sigma_n$, such that its composition with the homomorphism $\inn\co\kG^\pm(S)\to\Aut^\I(\kG(S))$ 
is the natural epimorphism $\kG^\pm(S)\to\Sigma_n$.
\end{lemma}

\begin{proof}We will first deal with the case $g\geq 1$ and $(g,n)\neq(1,2)$. We will then explain how to adapt the proof 
to the case $g=0$ and $n\geq 5$.

Let us fix a set of labels $B$ for the punctures on $S$. 
For a $1$-simplex $\{\g_0,\g_1\}\in C(S)$ such that the two curves $\g_0,\g_1$ bound a $1$-punctured annulus on $S$, we define the \emph{marked}
topological type of $\{\g_0,\g_1\}$ as the data of the topological type of $\{\g_0,\g_1\}$ plus the label $b\in B$ which marks the puncture on the annulus
with boundary $\g_0\cup\g_1$. Note that the hypothesis $S\neq S_{1,2}$ makes sure that this is well defined. 

There are $n$ marked topological types of such simplices and they are in bijective correspondence with the $\PG(S)$-orbits of $1$-simplices
$\{\g_0,\g_1\}$ of the given topological type. The action of $\G^\pm(S)$ on $C(S)$ then permutes these orbits and the natural epimorphism 
$\G^\pm(S)\to\Sigma_n$ can be recovered from this action.

For a $1$-simplex $\{\g_0,\g_1\}\in \kC(S)$ with the topological type of a $1$-punctured annulus on $S$ described above, we then define its 
\emph{marked} topological type as the marked topological type of any $1$-simplex $\{\g_0',\g_1'\}\in C(S)$ in the $\kPG(S)$-orbit of $\{\g_0,\g_1\}$.  
Since there holds $\kC(S)_1/\kPG(S)=C(S)_1/\PG(S)$, this is well defined. 

By \cite[Theorem~5.5]{BF}, the natural representation $\Aut^\I(\kG^\pm(S))\to\Aut(\kC(S))$ preserves the topological types of simplices 
in $\kC(S)$. This determines an action on marked topological types of $1$-simplices $\s\in\kC(S)$ with the topological type 
of a $1$-punctured annulus on $S$ if and only if, $\s'=f\cdot\s$, for some $f\in\kPG(S)$, implies that, for every 
$\phi\in\Aut^\I(\kG^\pm(S))$, there exists an $f'\in\kPG(S)$ such that $\phi\cdot\s'=f'\cdot\phi\cdot\s$.

The pure mapping class group $\PG(S)$ is the subgroup of the extended mapping class group $\G^{\pm}(S)$ generated by Dehn twists.
Hence, $\kPG(S)$ is an $\I$-characteristic subgroup of $\kG^{\pm}(S)$ in the sense of Definition~\ref{Ichar}. 
This implies that $\Inn(\kPG(S))$ is a normal subgroup of $\Aut^\I(\kG^{\pm}(S))$. Therefore, the natural representation 
$\Aut^\I(\kG^\pm(S))\to\Aut(\kC(S))$ satisfies the property stated above. The action of $\Aut^\I(\kG^\pm(S))$ on marked topological types of $1$-simplices 
of the topological type of a $1$-punctured annulus then determines a representation $\Aut^\I(\kG^\pm(S))\to\Sigma_n$ with the desired properties.

For $g=0$ and $n\geq 5$, we consider instead $1$-simplices $\{\g_0,\g_1\}\in C(S)$ such that the curves $\g_0,\g_1$ together bound a 
$1$-punctured annulus and, moreover, one of the two curves bounds a $2$-punctured disc on $S$. This is a well defined topological type 
of $1$-simplices for $n\geq 5$ which we then mark by labeling the puncture on the annulus with boundary $\g_0\cup\g_1$.
We can then proceed exactly as in the case $g\geq 1$ and $(g,n)\neq(1,2)$ treated above.
\end{proof}

\begin{lemma}\label{Icharacteristic}The group $\kPG(S)$ is an $\I$-characteristic subgroup of $\kPG^{\pm}(S)$
and, for $d(S)>1$ and $S\neq S_{1,2}$, $\kPG^{\pm}(S)$ is an $\I$-characteristic subgroup of $\kG^{\pm}(S)$.
\end{lemma}

\begin{proof}The pure mapping class group $\PG(S)$ is the subgroup of the extended mapping class group $\G^{\pm}(S)$ generated by Dehn twists.
Hence, $\kPG(S)$ is an $\I$-characteristic subgroup of both $\kG^{\pm}(S)$ and $\kPG^{\pm}(S)$.

In order to prove that, for $d(S)>1$ and $S\neq S_{1,2}$, $\kPG^{\pm}(S)$ is an $\I$-characteristic subgroup of $\kG^{\pm}(S)$, 
it is enough to show that $\Inn(\kPG^{\pm}(S))$ is a normal subgroup of $\Aut^\I(\kG^{\pm}(S))$. This immediately follows from the identity
$\Inn(\kPG^{\pm}(S))=\Inn(\kG^{\pm}(S))\cap\ker\Theta$, where $\Theta\co\Aut^\I(\kG^\pm(S))\to\Sigma_n$ is the representation
constructed in Lemma~\ref{autcurvecomplex}.
\end{proof}

\begin{remark}Lemma~\ref{Icharacteristic} fails for $S= S_{1,2}$. In fact, the exceptional automorphism $\exp\zeta\in\Aut^\I(\kG^{\pm}(S))$ 
constructed in Section~\ref{exAut} has the property that $\exp\zeta(\kPG^{\pm}(S))\neq\kPG^{\pm}(S)$. 
\end{remark}

From Lemma~\ref{Icharacteristic}, it follows that the elements of $\Aut^\I(\kPG^{\pm}(S))$ are compatible with the augmentation map
$\kPG^{\pm}(S)\to\Z/2$ and so preserve the set of antiholomorphic involutions. Moreover, by (ii) of Proposition~\ref{goodgroup}, for $g(S)\leq 2$,
the sets of conjugacy classes of antiholomorphic involutions in $\PG^{\pm}(S)$ and $\hPG^{\pm}(S)$ can be identified. We have:

\begin{lemma}\label{preserve}For $g(S)=0$, there is only one $\hG(S)^\pm$-conjugacy class of antiholomorphic involutions in $\hPG(S)^\pm$, on which 
$\Aut^\I(\hPG(S)^\pm)$ acts naturally.
\end{lemma} 

\begin{proof}Antiholomorphic involutions in $\PG(S)^\pm$ do not swap the punctures of $S$. Therefore, all the punctures lie in $\Fix(\iota)$, which,
in particular, is not empty. Since $g(S)=0$, $\Fix(\iota)$ is also connected and separating. Hence, an antiholomorphic involution 
in $\PG(S)^\pm$ is determined by the cyclic order of the punctures on $\Fix(\iota)$. Since $\G(S)^\pm$ acts transitively on such orderings, 
any two antiholomorphic involutions in $\PG(S)^\pm$ are conjugated by an element of $\G(S)^\pm$. 
The conclusion then follows from (ii) of Proposition~\ref{goodgroup}.
\end{proof}

\begin{lemma}\label{fixed}For $g(S)=0$, the fixed-point set $C(S)_0^\iota$ of an antiholomorphic involution $\iota\in\PG(S)^\pm$ is finite 
and consists of isotopy classes of simple closed curves on $S$ which have between them geometric intersection either $0$ or $2$.
\end{lemma} 

\begin{proof}The complement $S\ssm\Fix(\iota)$ is the disjoint union of two unpunctured discs. Let then $D'$ and $D''$ be closed subdiscs
of $S$ such that $D'\cap D''=\Fix(\iota)$. Fix on the surface $S$ a hyperbolic complete metric compatible with the involution $\iota$. 
The geodesic representative $\alpha$ of $\{[\alpha]\}\in C(S)_0^\iota$ for this metric has then the property that $\alpha=\iota(\alpha)$ and
the intersection $\alpha\cap D'$ is a disjoint union of geodesic arcs with boundary in $\Fix(\iota)$. If $\alpha'$ is one such arc, the union
$\alpha'\cup\iota(\alpha')\subset\alpha$ is a simple closed curve, which immediately implies that $\alpha'\cup\iota(\alpha')=\alpha$.
In particular, $\alpha$ crosses transversally $\Fix(\iota)$ in the two boundary points of $\alpha'$. The isotopy class of $\alpha$ is 
then determined by the partition which $\alpha'$ induces on the set of punctures of $S$ (which all lie on $\Fix(\iota)$). 
This implies the first statement of the proposition.

For $\{[\alpha]\}\neq\{[\beta]\}\in C(S)_0^\iota$, let also $\beta$ be the geodesic representative of the isotopy class $[\beta]$ for
the fixed metric. As above, we then have that $\beta=\iota(\beta)$ and the simple closed curve $\beta$ is the union of
the two arcs $\beta'=\beta\cap D'$ and $\beta''=\beta\cap D''$. It is also clear that the geodesic arcs $\alpha'$ and $\beta'$ on
the hyperbolic disc $D'$ have geometric intersection either $0$ or $1$. The second statement of the proposition then follows as well.
\end{proof}

By Lemma~\ref{Icharacteristic}, there is a natural homomorphism $\Aut^{\I}(\kG^{\pm}(S))\to\Aut^{\I}(\kPG^{\pm}(S))$, 
for $d(S)>1$ and $S\neq S_{1,2}$, and a natural homomorphism $\Aut^{\I}(\kPG^{\pm}(S))\to\Aut^{\I}(\kPG(S))$, 
for any hyperbolic surface $S$, both induced by restriction of automorphisms. We then have:

\begin{lemma}\label{injective}For $d(S)>1$, there holds:
\begin{enumerate} 
\item For $S\neq S_{1,2}$, the homomorphism $\Aut^{\I}(\kG^{\pm}(S))\to\Aut^{\I}(\kPG^{\pm}(S))$ is injective.
\item The kernel of the homomorphism $\Aut^{\I}(\kPG^{\pm}(S))\to\Aut^{\I}(\kPG(S))$ is trivial for $S\neq S_2$
and is generated by $\exp\zeta$ for $S=S_2$.
\end{enumerate}
\end{lemma}

\begin{proof}(i): The statement is trivial for $n(S)\leq 1$ so that we can assume that either $g(S)=1$ and $n(S)>2$ or $g(S)\geq 2$ and $n(S)>1$. 
In particular, by Proposition~\ref{slim}, we can assume that the centralizer $Z_{\kG(S)^\pm}(\kPG^\pm(S))$ is trivial. This implies that the restriction 
of the homomorphism $\Aut^{\I}(\kG^{\pm}(S))\to\Aut^{\I}(\kPG^{\pm}(S))$ to $\Inn(\kG^{\pm}(S))$ is injective, so that the conclusion now follows 
from \cite[Lemma~3.3]{BF}.
\smallskip

\noindent
(ii): Since $\Aut^{\I}(\kPG^{\pm}(S))\subseteq\Aut(\kPG^{\pm}(S))_{\kPG(S)}$, in order to determine the kernel
of the homomorphism $\Aut^{\I}(\kPG^{\pm}(S))\to\Aut^{\I}(\kPG(S))$, we can use the
Wells' exact sequence associated to the short exact sequence $1\to\kPG(S)\to\kPG^{\pm}(S)\to\Z/2\to 1$.

Since $\Aut(\Z/2)=\{1\}$, this implies that there is, in particular, a natural short exact sequence (cf.\ \cite[Lemma~7.4]{BF}):
\[1\to\Hom(\Z/2,Z(\kPG(S)))\stackrel{\exp}{\to}\Aut(\kPG^{\pm}(S))_{\kPG(S)}\to\Aut(\kPG(S)).\]

If $S\neq S_2$, then $Z(\kPG(S))=\{1\}$ and the conclusion follows. For $S=S_2$, we have $Z(\kPG(S))\cong\Z/2$ and
the only nontrivial element $\phi\in\Hom(\Z/2,Z(\kPG(S)))$ is mapped to the automorphism $\exp(\phi)$ of $\kPG^{\pm}(S)$ defined by the assignment
$x\mapsto x\cdot\phi(\mathrm{aug}(x))$, where $\mathrm{aug}\co\kPG^{\pm}(S)\to\Z/2$ is the orientation character.
This coincides with the exceptional automorphism $\exp\zeta$ defined in Section~\ref{exAut}. Hence, the second part of item (ii) follows.
\end{proof}

By Lemma~\ref{injective}, we then have:
 
\begin{lemma}\label{chain}For $d(S)>1$ and $S\neq S_{1,2}$, there is a chain of natural inclusions:
\[\Inn(\kG^{\pm}(S))\subseteq\Aut^{\I}(\kG^{\pm}(S))\subseteq\Aut^{\I}(\kPG^{\pm}(S))\]
and, for $d(S)>1$, a natural homomorphism $\Aut^{\I}(\kPG^{\pm}(S))\to\Aut^{\I}(\kPG(S))$, 
whose kernel is trivial for $S\neq S_2$ and is generated by $\exp\zeta$ for $S=S_2$.
\end{lemma}

\subsection{Proof of Theorem~\ref{autpants}}
(i): This item immediately follows from Theorem~\ref{completepantsrigidity}, for the case $d(S)>1$, and from \cite[Proposition~8.2]{BF}
for the case $d(S)=1$.
\medskip

\noindent
(ii) and (iii): For $S=S_{1,2}$, by item (ii) of Proposition~\ref{(g,n)=(1,2)}, there is a natural isomorphism 
$\Aut^\I(\hG^\pm(S))\cong\Aut^\I(\hPG^\pm(S))\times\langle\exp\zeta\rangle$. For $S=S_2$, the subgroups $\Inn(\kG^{\pm}(S))$ 
and $\langle\exp\zeta\rangle$ of $\Aut^{\I}(\kPG^{\pm}(S))$ are normal and have trivial intersection, so that they generate in 
$\Aut^{\I}(\kPG^{\pm}(S))$ a subgroup isomorphic to the direct product $\Inn(\kG^{\pm}(S))\times\langle\exp\zeta\rangle$. 

Therefore, by Lemma~\ref{chain}, in order to prove items (ii) and (iii) of Theorem~\ref{autpants}, 
it is enough to show that the images of $\Inn(\kG^{\pm}(S))$ and $\Aut^{\I}(\kPG^{\pm}(S))$ in $\Aut^{\I}(\kPG(S))$ coincide. 
The idea is to use Theorem~\ref{mainlemma}. We need to consider separately three different cases.

\subsection{Proof of Theorem~\ref{autpants}, (ii), for $g(S)=0$}\label{genus0case}
By Lemma~\ref{preserve}, after composing with an inner automorphism of $\kG(S)^\pm$, we can assume that a given element 
$f\in\Aut^\I(\kPG(S)^\pm)$ preserves a fixed antiholomorphic involution $\iota\in\PG(S)^\pm$ whose fixed-point set
in the surface $S$ is a separating simple closed curve containing all punctures of $S$.

We will now make use of the results in Section~\ref{fpset}. 
By Lemma~\ref{fixed}, the fixed-point set $C(S)^\iota_0$, for the action of $\iota$ on $C(S)_0$, is finite and then, by 
Proposition~\ref{boundaryfixedpro}, identifies with the fixed-point set $\kC(S)^\iota_0$, for the action of $\iota$ on $\kC(S)_0$. 

Let $\{\alpha,\beta\}$ be a pair of $\iota$-invariant simple closed curves on $S$ which intersect precisely in two points.
An element  $f\in\Aut^\I(\kPG^{\pm}(S))$, such that $f(\iota)=\iota$, then preserves $C(S)^\iota_0=\kC(S)^\iota_0$ and 
sends the pair $\{[\alpha],[\beta]\}$ of $0$-simplices in $C(S)^\iota_0$ to the pair $\{f([\alpha]),f([\beta])\}$ also contained in $C(S)^\iota_0$.
Since $f([\alpha])$ and $f([\beta])$ cannot have trivial geometric intersection, otherwise $f$ would not preserve the simplicial structure
of $\kC(S)$, from Lemma~\ref{fixed}, it follows that $f([\alpha])$ and $f([\beta])$ have geometric intersection $2$. 

We can then complete the two $0$-simplices $\{[\alpha]\},\{[\beta]\}\in C(S)^\iota_0$ to two $(n(S)-4)$-simplices $v_\alpha,v_\beta$ of $C(S)$ 
whose sets of vertices coincide except for the elements $[\alpha]\in v_\alpha$ and $[\beta]\in v_\beta$. In this way, we have defined an edge 
$\{v_\alpha,v_\beta\}$ of the pants complex $C_P(S)\subset\kC_P(S)$ with the property that $\{f(v_\alpha),f(v_\beta)\}$ is also 
an edge of $\kC_P(S)$. By Theorem~\ref{mainlemma}, we conclude that $f\in\Inn(\kG^{\pm}(S))$.

\subsection{Proof of Theorem~\ref{autpants}, (iii), for $S=S_{1,2}$}\label{genus12case}
We will need an analogue, for extended mapping class groups, of Lemma~\ref{S_{1,2}}:

\begin{lemma}\label{S_{1,2}bis}The group $\kPG^\pm(S_{0,5})$ identifies with an $\I$-characteristic subgroup of $\kPG^\pm(S_{1,2})$.
In particular, there is a natural homomorphism $\Aut^\I(\kPG^\pm(S_{1,2}))\to\Aut^\I(\kPG^\pm(S_{0,5}))$ induced by restriction of automorphisms.
\end{lemma}

\begin{proof}As we remarked in Section~\ref{05case}, if we let $S_{/\u}$ be the quotient of the surface $S_{1,2}$ by 
the hyperelliptic involution $\u$ and let $B_\u$ be the branch locus of the orbit map $S\to S_{/\u}$, the surface $S_{/\u}$ is a $1$-punctured sphere and
there is a diffeomorphism $S_{/\u}\ssm B_\u\cong S_{0,5}$. If we denote by $Q$ the puncture of $S_{0,5}$ which corresponds to the puncture of $S_{/\u}$,
there is then a natural isomorphism $\kPG^\pm(S_{1,2})\cong\kG^\pm(S_{0,5})_Q$. In order to prove the lemma, it is then enough to prove that
$\kPG^\pm(S_{0,5})$ is an $\I$-characteristic subgroup of $\kG^\pm(S_{0,5})_Q$.

Since $\kPG(S_{0,5})$ is an $\I$-characteristic subgroup of $\kG^\pm(S_{0,5})_Q$, the same argument of proof of 
Lemma~\ref{autcurvecomplex} shows that there is a natural homomorphism $\Theta\co\Aut^\I(\kG^\pm(S_{0,5})_Q)\to\Sigma_B$, where $B$ 
is a set of labels for the punctures of $S_{0,5}$.

The symmetric group $\Sigma_{B\ssm Q}$ identifies with the stabilizer of $Q\in B$ for the tautological action of $\Sigma_B$ on $B$. 
The fact that $\Inn(\kG^\pm(S_{0,5})_Q)$ is a normal subgroup of $\Aut^\I(\kG^\pm(S_{0,5})_Q)$ and
$\Theta(\Inn(\kG^\pm(S_{0,5})_Q))=\Sigma_{B\ssm Q}$ then implies that $\Theta(\Aut^\I(\kG^\pm(S_{0,5})_Q))=\Sigma_{B\ssm Q}$ as well.

As in the proof of Lemma~\ref{Icharacteristic}, we conclude that $\Inn(\kPG^\pm(S_{0,5}))=\Inn(\kG^\pm(S_{0,5})_Q)\cap\ker\Theta$
is a normal subgroup of $\Aut^\I(\kG^\pm(S_{0,5})_Q)$, which implies that $\kPG^\pm(S_{0,5})$ is an $\I$-characteristic subgroup of $\kG^\pm(S_{0,5})_Q$.
\end{proof}

By (ii) of Lemma~\ref{injective} and Lemma~\ref{S_{1,2}}, restriction of automorphisms induces a natural monomorphism 
$\Aut^\I(\kPG^\pm(S_{1,2}))\hookra\Aut^\I(\kPG(S_{0,5}))$. By Lemma~\ref{S_{1,2}bis}, this homomorphism then factors through 
the natural homomorphisms $\Aut^\I(\kPG^\pm(S_{1,2}))\to\Aut^\I(\kPG^\pm(S_{0,5}))$ and  
$\Aut^\I(\kPG^\pm(S_{0,5}))\to\Aut^\I(\kPG(S_{0,5}))$, also both induced by restriction of automorphisms.

By the case $S=S_{0,5}$ of the theorem, treated above, we then have that 
the image of the group $\Aut^\I(\kPG^\pm(S_{1,2}))$ in $\Aut^\I(\kPG(S_{0,5}))$ is contained in the subgroup $\Inn(\kG^{\pm}(S_{0,5}))$.
This implies this case of the theorem, since, as in the proof of Theorem~\ref{mainlemma} for $S=S_{1,2}$, by \cite[Lemma~9.13]{BF}, the inner 
automorphisms of $\kG^{\pm}(S_{0,5})$ which normalize the subgroup $\kG^\pm(S_{0,5})_Q\equiv\kPG^\pm(S_{1,2})$ belong to $\Inn(\kPG^\pm(S_{1,2}))$.

\subsection{Proof of Theorem~\ref{autpants}, (ii) and (iii), for $g(S)\geq 1$ and $d(S)>2$}The hypotheses implies that, 
for a nonseparating simple closed curve $\g$ on $S$, we have $d(S\ssm\g)>1$. We then proceed by
induction on the genus of $S$, where the base for the induction is provided by Section~\ref{genus0case}.

Given an element $f\in\Aut^\I(\kPG^{\pm}(S))$, let us consider its action on $\kC(S)$. By \cite[Theorem~5.5]{BF}, after possibly composing
with an inner automorphism of $\kPG(S)$, we may assume that $f\in\Aut^\I(\kPG(S))_{\hI_\g}$, where $\hI_\g$ is the procyclic subgroup associated 
to some nonseparating simple closed curve $\g$ on $S$. In particular, the automorphism $f$ preserves the star $\Star(\g)$ and, by Lemma~\ref{factor}, 
acts on the link $\Link(\g)\cong \kC(S_\g)$ through its image in the group $\Aut^\I(\kPG^{\pm}(S_\g)$. Thus, $f$ preserves the vertex set of the 
subgraph $L_\g$ of $\kC_P(S)$, which identifies with the profinite set of $(d(S)-1)$-simplices of $\Star(\g)$, 
and acts continuously on this set through an element of $\Aut^\I(\kPG^{\pm}(S_\g))$.
Since, by Lemma~\ref{linkpants}, $L_\g\cong\kC_P(S_\g)$, the induction hypothesis implies that $f$ also preserves the edge set of $L_\g$.
By Theorem~\ref{mainlemma}, we then conclude that the image of $f$ in $\Aut^\I(\kPG(S))$ is contained in $\Inn(\kG^{\pm}(S))$.

\section{Proof of Corollary~\ref{selfnorm}}
After possibly replacing $\kG(S)$ by $\kG(S)/Z(\kG(S))$ (which is also isomorphic to a procongruence mapping class group), 
we can assume that the center $Z(\kG(S))$ of $\kG(S)$ is trivial. 

Let $f\in\Aut^\I(\kG(S))$ be an automorphism which normalizes the subgroup $\Inn(\kG^\pm(S))$. This implies that $f$ satisfies the second condition
of Lemma~\ref{Wells}. Since the first one is satisfied by the above assumption and the third one is trivially satisfied, we have that
$f$ extends to an automorphism of $\kG^\pm(S)$. The conclusion then follows from Theorem~\ref{autpants}.

\section{Proof of Theorem~\ref{genus0}}
By Theorem~\ref{autpants}, in order to prove Theorem~\ref{genus0}, we have to show that, for $g(S)=0$, there holds 
$\Aut(\hG^{\pm}(S))=\Aut^\I(\hG^{\pm}(S))$ and $\Aut(\hPG^{\pm}(S))=\Aut^\I(\hPG^{\pm}(S))$. For this, we need to
establish first that $\hPG(S)$ is a characteristic subgroup of both $\hPG^{\pm}(S)$ and $\hG^{\pm}(S)$. 
For this, we will need a series of lemmas.

The first part of the statement of the following lemma is well-known:

\begin{lemma}\label{mpowers}For $n\geq 5$, there are unique surjective homomorphisms $\G(S_{0,n})\to \Sigma_n$ and
$\G^{\pm}(S_{0,n})\to \Sigma_n$, up to automorphisms of $\Sigma_n$.
\end{lemma} 

\begin{proof}The first part of the statement easily follows from classical results by Artin in \cite{Artin} (cf.\ \cite{Kolay}, for an 
alternative proof). Let $G$ be a group which acts transitively on a set of $n$ letters and is generated by $n-1$ elements satisfying the standard braid relations. 
By (cf.\ \cite[Theorem~3]{Artin} and its proof), for $n\geq 4$, there is a unique epimorphism $G\to \Sigma_n$, up to automorphisms of $\Sigma_n$. 
This result applies, in particular, to the mapping class group $\G(S_{0,n})$, from which, the first statement of the lemma follows.

Let us now prove that, up to automorphisms of $\Sigma_n$, there is a unique epimorphism $\G^{\pm}(S_{0,n})\to \Sigma_n$. 
By \cite[Lemma~6]{Artin}, a group $G$ as above does not admit an epimorphism to the alternating group $A_n$, for $n\geq 5$. 
This implies that there is no epimorphism $\G(S_{0,n})\to A_n$, for $n\geq 5$. Hence, every epimorphism $\G^{\pm}(S_{0,n})\to \Sigma_n$ 
restricts to an epimorphism $\G(S_{0,n})\to \Sigma_n$. From the first part of the proof, it then follows that the kernel of any epimorphism 
$\G^{\pm}(S_{0,n})\to \Sigma_n$ contains the pure mapping class group $\PG(S_{0,n})$.

\begin{lemma}\label{quotientiso}There is a natural isomorphism $\G^{\pm}(S_{0,n})/\PG(S_{0,n})\cong\Sigma_n\times\Z/2$.
\end{lemma}

\begin{proof}The natural homomorphisms $\G^{\pm}(S_{0,n})\to\Sigma_n$ and $\G^{\pm}(S_{0,n})\to\Z/2$
determine a homomorphism $\phi\co\G^{\pm}(S_{0,n})\to\Sigma_n\times\Z/2$, which maps the subgroups $\G(S_{0,n})$ and
$\PG^{\pm}(S_{0,n})$ of $\G^{\pm}(S_{0,n})$, respectively, onto the subgroups $\Sigma_n\times\{1\}$ and $\{1\}\times\Z/2$ of $\Sigma_n\times\Z/2$.
Hence, the homomorphism $\phi$ is surjective. Since the kernel of $\phi$ contains $\PG(S_{0,n})$ and the groups $\G^{\pm}(S_{0,n})/\PG(S_{0,n})$ 
and $\Sigma_n\times\Z/2$ have the same order, the conclusion follows.
\end{proof}

The second claim of the lemma now follows from Lemma~\ref{quotientiso} and the fact that the group $\Sigma_n\times\Z/2$ 
admits a unique epimorphism to $\Sigma_n$, up to automorphisms of the latter group.
\end{proof}

We also have the following group-theoretic lemma (cf.\  \cite[Lemma~2.3]{MN}):

\begin{lemma}\label{charcriterion}Let $G$ be a finitely generated group and $V$ a finite index normal subgroup with the property that 
all epimorphisms from $G$ to the quotient group $G/V$ have the same kernel $V$.
Then, the closure $\wh{V}$ of the image of $V$ in the profinite completion $\wh{G}$ of $G$ is an open characteristic subgroup.
\end{lemma}

\begin{proof}Since $G$ is finitely generated, by a classical result of Nikolov and Segal, any epimorphism $\wh{G}\to G/V$ is continuous 
and so restricts to an epimorphism $G\to G/V$, which, by our hypothesis, has kernel $V$. Hence, all epimorphisms from $\wh{G}$ to $G/V$ 
have the same kernel $\wh{V}$, which shows that $\wh{V}$ is indeed a characteristic subgroup of $\wh{G}$.
\end{proof}

\begin{lemma}\label{index2}For $n\geq 4$, let $\G$ be an index $2$ subgroup of $\PG^\pm(S_{0,n})$ distinct from $\PG(S_{0,n})$.
Then, for every prime $p>2$, there holds $\dim H^1(\G,\F_p)<\dim H^1(\PG(S_{0,n}),\F_p)$.
\end{lemma}

\begin{proof}Let $\G':=\G\cap\PG(S_{0,n})$ and $C:=\G/\G'$. Then, $C$ is cyclic of order $2$ and, for every prime $p>2$, there holds
$H^1(\G,\F_p)=H^1(\G',\F_p)^C$. Since $\PG(S_{0,n})$ is generated by Dehn twists about simple closed curves bounding a $2$-punctured disc,
there is a puncture on $S_{0,n}$ such that the kernel $\pi_1(S_{0,n-1})$ of the epimorphism $q\co\PG^\pm(S_{0,n})\to\PG^\pm(S_{0,n-1})$, obtained filling in 
this puncture, is not contained in $\G'$. This implies that $q(\G')=\PG(S_{0,n-1})$ (and also $q(\G)=\PG^\pm(S_{0,n-1})$), so that there is a short exact sequence:
\[1\to K\to\G'\stackrel{q}{\to}\PG(S_{0,n-1})\to 1,\]
where $K:=\G'\cap\pi_1(S_{0,n-1})$ is a subgroup of index $2$ of $\pi_1(S_{0,n-1})$. By the associated inflation-restriction exact sequence, after
taking $C$-invariants, we get the exact sequence:
\[0\to H^1(\PG(S_{0,n-1}),\F_p)^C\to H^1(\G',\F_p)^C\to H^1(K,\F_p)^C.\]

The group $C$ is generated by the image of an orientation reversing element $f\in\PG^\pm(S_{0,n})$ and
$H^1(\PG(S_{0,n-1}),\F_p)^C=H^1(\PG(S_{0,n-1}),\F_p)^{q(f)}$. For every Dehn twist $\tau_\g\in\PG(S_{0,n-1})$, there holds
$q(f)(\tau_\g)p(f)^{-1}=\tau^{-1}_{q(f)}$. Hence, $q(f)$ acts on the abelianization $\PG(S_{0,n-1})^\mathrm{ab}$ by multiplication by $-1$,
which implies that $H^1(\PG(S_{0,n-1}),\F_p)^C=\{0\}$.

By the standard description of a double unramified covering of a punctured sphere, 
the group $K$ has a presentation $K=\langle u_1^2,u_2^2,t_1,\dots,t_{2(n-3)}|\prod u_1^2 u_2^2 t_1\dots t_{2(n-3)}=1\rangle$,
where $u_1,u_2\in\pi_1(S_{0,n-1})$ are freely homotopic to peripheral simple closed curves on $S_{0,n-1}$.

The abelianization $K^\mathrm{ab}$ is then described as the abelian group freely generated by the images
$\bar u_1^2,\bar u_2^2,\bar t_1,\dots,\bar t_{2(n-3)}$ of the standard generators of $K$ in $K^\mathrm{ab}$ subjected 
to the only relation $\bar u_1^2+\bar u_2^2+\bar t_1+\dots+\bar t_{2(n-3)}=0$.
 
Let us now consider the action (induced by conjugation) of $f$ on $K^\mathrm{ab}$. Since the action of $\PG^\pm(S_{0,n-1})$ on 
$S_{0,n-1}$ fixes peripheral simple closed curves changing their orientation according to the orientation character, we have that the induced action of 
$f$ on $K^\mathrm{ab}$ restricts on the primitive submodule generated by $\bar u_1^2$ and $\bar u_2^2$ to multiplication by $-1$. This implies that 
there is an inequality $\dim H^1(K,\F_p)^C\leq\dim H^1(K/\langle\bar u_1^2,\bar u_2^2\rangle,\F_p)<2(n-3)$, so that, in conclusion, we obtain the inequality 
$\dim H^1(\G',\F_p)^C=\dim H^1(\G,\F_p)<2(n-3)$.

On the other hand, it is well known that $\dim H^1(\PG(S_{0,n}),\F_p)=(n-1)(n-2)/2 -1\geq 2(n-3)$, for $n\geq 4$, and the lemma follows.
\end{proof}

We can now prove the following refinement of \cite[Proposition~4.1, (ii)]{MN}:

\begin{theorem}\label{characteristic}For $g(S)= 0$, the profinite pure mapping class group $\hPG(S)$ is a characteristic subgroup of $\hG(S)$,
$\hPG^{\pm}(S)$ and $\hG^{\pm}(S)$.
\end{theorem}

\begin{proof}For $n(S)=4$, we only need to observe that $\hPG(S)$ is the maximal normal free profinite subgroup contained 
in all the groups in the statement of the proposition. For $n(S)\geq 5$, the statement that $\hPG(S)$ (resp.\ $\hPG^\pm(S)$) is a characteristic subgroup 
of $\hG(S)$ (resp.\ $\hG^\pm(S)$) follows from Lemma~\ref{mpowers} and Lemma~\ref{charcriterion}. 

In order to complete the proof of the theorem, we only need to show that $\hPG(S)$ is a characteristic subgroup of $\hPG^{\pm}(S)$.
For every group $G$, there holds $H^1(\wh{G},\F_p)\cong H^1(G,\F_p)$. 
This, together with Lemma~\ref{index2}, implies that $\hPG(S)$ is characteristic in $\hPG^{\pm}(S)$.
\end{proof}

The following result does not appear anywhere in the literature and is thus of independent interest:

\begin{theorem}\label{decpreservation}For $g(S)= 0$ and $n(S)\geq 5$, we have $\Aut^\I(\hPG(S))=\Aut(\hPG(S))$.
\end{theorem}

\begin{proof}The following weak version of Theorem~\ref{decpreservation} is essentially a consequence of \cite[Corollary~C]{HMM}
and \cite[Main Theorem]{HS}:

\begin{lemma}\label{HosMinMoch}For $g(S)= 0$ and $n(S)\geq 5$, every automorphisms of $\hPG(S)$ preserves the set of procyclic inertia groups
topologically generated by the profinite Dehn twists of topological type a simple closed curve bounding a $2$-punctured disc on $S$.
\end{lemma}

\begin{proof}Let us translate some of the terminology and results from \cite{HMM} in our setting. Let $\ol{S}$ be the surface obtained from $S$ 
filling in $1\leq n'< n(S)-3$ of the $n(S)$ punctures of $S$. There is then an associated epimorphism of mapping class groups
$p_{\ol{S}}\co\PG(S)\to\PG(\ol{S})$ and so of profinite mapping class groups $\hat{p}_{\ol{S}}\co\hPG(S)\to\hPG(\ol{S})$.
The subgroup $\ker\hat{p}_{\ol{S}}$ of $\hPG(S)$ is what is called in \cite{HMM} a \emph{generalized fiber subgroup} (cf.\ Definition~2.1 ibid.).

In \cite[Definition~2.1]{HMM}, the group $\Out^\mathrm{gF}(\hPG(S))$\index{$\Out^\mathrm{gF}(\hPG(S))$,  the subgroup of outer automorphisms which preserve all generalized fiber subgroups} is then defined to be the subgroup of $\Out(\hPG(S))$ 
consisting of those automorphisms which preserve \emph{all} generalized fiber subgroups. In particular, there is also an induced homomorphism:
\[\hat{q}_{\ol{S}}\co\Out^\mathrm{gF}(\hPG(S))\to\Out(\hPG(\ol{S})).\]

Note that there is an outer action of the symmetric group $\Sg_{n(S)}\cong\hG(S)/\hPG(S)$ 
on $\hPG(S)$ induced by restriction of the inner automorphisms of $\hG(S)$ to $\hPG(S)$. In this way, $\Sg_{n(S)}$ is also identified with a 
subgroup of $\Out(\hPG(S))$ and it is easy to check that there holds:
\[\Out^\mathrm{gF}(\hPG(S))\cap\Sg_{n(S)}=\{1\}.\]

The statement of \cite[Corollary~C]{HMM} can then be parsed, in the above terminology, into the following series of statements 
(cf.\ \cite[Corollary~2.6 and Corollary~2.8]{HMM}):
\begin{enumerate}
\item there holds $Z_{\Out(\hPG(S))}(\Out^\mathrm{gF}(\hPG(S)))=\Sg_{n(S)}$;
\item there is a direct product decomposition $\Out(\hPG(S))=\Out^\mathrm{gF}(\hPG(S))\times\Sg_{n(S)}$;
\item for $n(\ol{S})=5$, the homomorphism $\hat{q}_{\ol{S}}$ is injective and identifies $\Out^\mathrm{gF}(\hPG(S))$
with the Grothendieck-Teichm\"uller group $\GT\subset\Out(\hPG(\ol{S}))$ as defined in \cite{HS} (cf.\ \cite[Definition~2.7]{HMM}).
\end{enumerate}

For $\g$ a simple closed curve bounding a $2$-punctured disc on $S$ containing a puncture labeled by $P$, 
the Dehn twist $\tau_\g\in\PG(S)$ is the image, by the push map $\pi_1(\ol{S}, P)\hookra\PG(S)$, of a simple loop around some puncture
of $\ol{S}$, where $\ol{S}$ is the surface obtained from $S$ filling in the puncture $P$ with a point labeled by the same letter.
It is clear that the image of the push map is generated by such elements and that the image of the profinite push map 
$\hpi_1(\ol{S}, P)\hookra\hPG(S)$ coincides with the generalized fiber subgroup $\ker\hat{p}_{\ol{S}}$ defined above.

From the definition of the group $\Out^\sharp(\hPG(S))$ in \cite{HS} (cf.\ Section~\ref{mainlemma=1}), it then easily follows that this group 
is contained in $\Out^\mathrm{gF}(\hPG(S))$. 

In particular, for $n(\ol{S})=5$, the epimorphism $\hat{p}_{\ol{S}}\co\hPG(S)\to\hPG(\ol{S})$
induces a homomorphism $\Out^\sharp(\hPG(S))\to\Out^\sharp(\hPG(\ol{S}))$ which, according to \cite[Main Theorem]{HS} and its proof, 
is an isomorphism and identifies $\Out^\sharp(\hPG(S))$ with the Grothendieck-Teichm\"uller group $\GT:=\Out^\sharp(\hPG(\ol{S}))$.

Thus, combining the results of \cite{HMM} and \cite{HS} exposed above, we have that:
\[\Out^\sharp(\hPG(S))=\Out^\mathrm{gF}(\hPG(S)).\]

By \cite[Corollary~2.6, (i)]{HMM}, we then have that given an automorphism $f\in\Aut(\hPG(S))$, after possibly composing it with the restriction 
of an inner automorphism of $\hG(S)$, we can assume that the image $\ol{f}$ of $f$ in $\Out(\hPG(S))$ is contained in the subgroup
$\Out^\sharp(\hPG(S))$ of $\Out(\hPG(S))$, which implies the lemma.
\end{proof} 

For $n(S)= 5$, every essential simple closed curve on $S$ bounds a $2$-punctured disc. Therefore, in this case, Lemma~\ref{HosMinMoch}
directly implies that $\Aut^\I(\hPG(S))=\Aut(\hPG(S))$. For $n(S)> 5$, the proof of the theorem proceeds by induction on $n(S)$. 
Let us then assume that the theorem holds for $n(S)= k-1$, where $n(S)\geq 6$, and let us prove it for $n(S)= k$.

Let us show that the action of $\Aut(\hPG(S))$, for $n(S)= k$, preserves the set of all inertia groups $\{\hI_\s\}_{\s\in\hC(S)}$ of $\hPG(S)$. 
As a first step, let us prove that:

\begin{lemma}\label{dualgraphpres}$\Aut(\hPG(S))$ preserves the subsets $\{\hI_\s\}_{\s\in\hC(S)_h}$, for $h=k-4,k-5$.
\end{lemma}

\begin{proof}It is enough to prove that, for $f\in\Aut(\hPG(S))$ and $\s\in C(S)_h$, for $h=k-4,k-5$, we have $f(\hI_\s)=\hI_{\s''}$, 
for some $\s''\in\hC(S)_h$. The hypothesis on $h$ implies that there is a simple closed curve $\beta\in\s$ on $S$ bounding a $2$-punctured disc and 
such that $\hI_\s\subset\hPG(S)_\beta$. By Lemma~\ref{HosMinMoch}, after composing $f$ with some inner automorphism, 
we can then assume that $f$ preserves the subgroup $\hPG(S)_\beta$.

By \cite[Theorem~4.10]{BF} (cf.\ the proof of Lemma~\ref{factor}), the latter group is described by the short exact sequence:
\[1\to\hI_\beta\to\hPG(S)_\beta\to\hPG(S_\beta)\to 1,\]
where $S_\beta$ is the connected component of $S\ssm\beta$ which contains more than two punctures.

In particular, since the center of $\hPG(S_\beta)$ is trivial, we have that $Z(\hPG(S)_\beta)=\hI_\beta$. Therefore, $f$ induces an automorphism
$\bar f$ of the quotient group $\hPG(S)_\beta/\hI_\beta\cong\hPG(S_\beta)$. 

Let us denote by $\bar I_\s\cong\hI_\s/\hI_\beta$ the image of $\hI_\s$ 
in the quotient. Note that $\bar I_\s$ identifies with an inertia group of $\hPG(S_\beta)$. By the inductive hypothesis, we then have that 
$\bar f(\bar I_\s)=\hI_{\s'}$, for some $\s'\in\hC(S_\beta)_{h-1}$. The embedding $S_\beta\subset S$ induces a continuous map of profinite sets 
$\hC(S_\beta)_{h-1}\to\hC(S)_{h-1}$. Hence, if we let $\tilde\s'$ be the image of $\s'$ in $\hC(S)_{h-1}$, for $\s'':=\tilde\s'\cup\beta$, 
there holds $f(\hI_\s)=\hI_{\s''}$, which proves the lemma.
\end{proof}

In Remark~\ref{groupthreal}, we identified the profinite curve complex $\hC(S)$ with the abstract simplicial 
profinite complex $\hC_\cI(S)$ whose set of $h$-simplices is the set of closed subgroups $\{\hI_\s\}_{\s\in\hC_h(S)}$.
The dual graph $\hC^\ast_\cI(S)$\index{$\hC^\ast_\cI(S)$, dual graph of $\hC_\cI(S)$} of $\hC_\cI(S)$ (cf.\ \cite[Definition~3.9]{BF})
has for vertex and edge sets, respectively, the sets $\{\hI_\s\}_{\s\in\hC(S)_h}$, for $h=k-4,k-5$. Thus, by Lemma~\ref{dualgraphpres}, 
the natural action of $\Aut(\hPG(S))$ on the set of closed subgroups of $\hPG(S)$ induces a representation $\Aut(\hPG(S))\to\Aut(\hC^\ast(S))$. 

By \cite[Lemma~6.5]{BF}, the profinite curve complex $\hC(S)$ can be reconstructed from its dual graph $\hC^\ast(S)$ 
and there is a natural isomorphism $\Aut(\hC^\ast(S))\cong\Aut(\hC(S))$. It then follows that the action of $\Aut(\hPG(S))$ 
on the set of closed subgroups preserves the set of \emph{all} inertia groups of $\hPG(S)$. 
\end{proof}

We now have the following corollary of Theorem~\ref{decpreservation}, which, as observed above, implies Theorem~\ref{genus0}:

\begin{corollary}\label{inertiacond}For $g(S)= 0$ and $n(S)\geq 5$, there holds $\Aut(\hG^{\pm}(S))=\Aut^\I(\hG^{\pm}(S))$ and 
$\Aut(\hPG^{\pm}(S))=\Aut^\I(\hPG^{\pm}(S))$.
\end{corollary}

\begin{proof}Let us make the trivial remark that, if an automorphism $f\in\Aut(\hPG(S))$ preserves the set $\{\hI_\s\}_{\s\in\hC(S)}$ of all inertia groups
(which are indeed contained in $\hPG(S)$), and it extends to an automorphism $\tilde f\in\Aut(\hG^{\pm}(S))$ (resp.\ $\tilde f\in\Aut(\hPG^{\pm}(S))$), 
then the automorphism $\tilde f$ also preserves the set $\{\hI_\s\}_{\s\in\hC(S)}$.

By Theorem~\ref{characteristic}, every automorphism of $\hG^{\pm}(S)$ (resp.\ $\hPG^{\pm}(S)$) restricts to an automorphism of
$\hPG(S)$. The corollary then follows from Theorem~\ref{decpreservation} and the above remark.
\end{proof}

\section{Proof of Corollary~\ref{absanabelian}}
The following lemma is well known. We give a proof for lack of a suitable reference:

\begin{lemma}\label{realetale}For a suitable choice of base point, the \'etale fundamental group of $(\cM_{g,n})_\R$ identifies with the extended
profinite mapping class group $\hPG^\pm(S_{g,n})$.
\end{lemma}
 
\begin{proof}A morphism $\Spec\R\to(\cM_{g,n})_\R$ is equivalent to the datum of a smooth $n$-punctured, genus $g$ curve $C$
defined over $\R$. Let $(C_\C,\iota)$ be the corresponding real complex algebraic curve and let us denote by $\xi\co\Spec\C\to(\cM_{g,n})_\R$
the geometric base point associated to the standard embedding $\R\subset\C$. 

Complex conjugation defines a (nontrivial) automorphism $\u$ of the geometric base point $\xi$ and then a (nontrivial) element, 
which we also denote by $\u$, of the \'etale fundamental group $\pi_1^\mathrm{et}((\cM_{g,n})_\R,\xi)$. The element $\u$, 
together with the image of the geometric \'etale fundamental group $\pi_1^\mathrm{et}((\cM_{g,n})_\C,\xi)$ in $\pi_1^\mathrm{et}((\cM_{g,n})_\R,\xi)$ 
(cf.\ the short exact sequence~ \eqref{fundexsequence}) generates the latter group. More precisely, after identifying the geometric 
\'etale fundamental group with its image in the \'etale fundamental group, there holds:
\[\pi_1^\mathrm{et}((\cM_{g,n})_\R,\xi)=\pi_1^\mathrm{et}((\cM_{g,n})_\C,\xi)\rtimes\langle\u\rangle.\]

Let us fix a homeomorphism $\phi\co S_{g,n}\stackrel{\sim}{\to}C$. The point $[\phi]\in\cT(S_{g,n})$ lies above $\xi$ and so determines 
an isomorphism $\pi_1((\cM_{g,n})_\C,\xi)\cong\PG(S_{g,n})$ and then $\pi_1^\mathrm{et}((\cM_{g,n})_\C,\xi)\cong\hPG(S_{g,n})$.
Let $\td\u$ be the mapping class of $\phi^{-1}\circ\iota\circ\phi$. This is an antiholomorphic involution of $\PG^\pm(S_{g,n})$ and there holds:
\[\PG^\pm(S_{g,n})=\PG(S_{g,n})\rtimes\langle\td\u\rangle\hspace{0.4cm}\mbox{and then also}\hspace{0.4cm}
\hPG^\pm(S_{g,n})=\hPG(S_{g,n})\rtimes\langle\td\u\rangle.\]

The action of $\u$ on the base point $\xi$ extends to an action on the complex moduli stack $(\cM_{g,n})_\C$. 
It is easy to check (for instance, by looking at the induced actions on the respective tangent spaces at $\xi$ and $[\phi]$) 
that the action of $\td\u$ on the universal covering $\cT(S_{g,n})$ is a lift the action of $\u$ on $(\cM_{g,n})_\C$. 
This implies that the action of $\u$ on $\pi_1^\mathrm{et}((\cM_{g,n})_\C,\xi)$ is compatible with the action of $\td\u$
on $\hPG(S_{g,n})$ via the isomorphism $\pi_1^\mathrm{et}((\cM_{g,n})_\C,\xi)\cong\hPG(S_{g,n})$.
Identifying $\u$ and $\td\u$ then gives the desired isomorphism $\pi_1^\mathrm{et}((\cM_{g,n})_\R,\xi)\cong\hPG^\pm(S_{g,n})$.
\end{proof}

Corollary~\ref{absanabelian} now immediately follows from Lemma~\ref{realetale}, Theorem~\ref{genus0}, the natural isomorphisms~\eqref{sigmaiso} 
and the fact that, by Royden's Theorem (cf.\ \cite[Section~2, Theorem]{EK}), for all $n\geq 5$, we have: 
\[\Aut_\R((\cM_{0,n})_\R)=\Aut((\cM_{0,n})_\C)=\Sigma_n.\]

\printindex

\end{document}